\documentclass[12pt]{scrartcl}
 \usepackage{cmap}
 \usepackage{amsmath,amsxtra,amscd,amssymb,latexsym,stmaryrd,amsthm}
 \usepackage{mathrsfs, calc, xcolor, xfrac}

  \usepackage{lmodern}
  \usepackage[utf8]{inputenc}
  \usepackage[T1]{fontenc}

\usepackage[allbordercolors=white,colorlinks=true,backref]{hyperref}

\addtokomafont{title}{\rmfamily\itshape}

\theoremstyle{plain}
\newtheorem{theo}{Theorem}[section]
\newtheorem*{theo*}{Theorem}
\newtheorem{coro}[theo]{Corollary}
\newtheorem{prop}[theo]{Proposition}
\newtheorem{lemm}[theo]{Lemma}
\newtheorem{theomain}{Theorem}
\newtheorem{coromain}[theomain]{Corollary}

\theoremstyle{definition}
\newtheorem{defi}[theo]{Definition}
\newtheorem{rema}[theo]{Remark}
\newtheorem{exem}[theo]{Example}

\newtheoremstyle{step}
{\topsep}
{.3em}
{}
{0pt}
{\itshape}
{:}
{ }
{\thmname{#1}\thmnumber{ #2}\thmnote{ (#3)}}

\theoremstyle{step}
\newtheorem{step}{Step}[theo]

\renewcommand{\ge}{\geqslant}
\renewcommand{\le}{\leqslant}

\newcommand*{\dd}%
  {\relax\ifnum\lastnodetype>0\mskip\medmuskip\fi\mathrm{d}}
\newcommand{\fspace}[1]{\mathcal{#1}}
\newcommand{\fX}{\fspace{X}}

\newcommand{\op}[1]{\mathscr{#1}}
\newcommand{\one}{\boldsymbol{1}}

\newcommand{\mchain}{\mathsf}
\newcommand{\proba}{\operatorname{\mathcal{P}}}

\newcommand{\SG}{\operatorname{\mathrm{SG}}}

\newcommand{\diam}{\operatorname{diam}}
\newcommand{\wass}{\operatorname{W}}

\newcommand{\Id}{\operatorname{Id}}

\newcommand{\Hol}{\operatorname{Hol}}

\newcommand{\C}{\mathscr{C}}
\newcommand{\hl}{\operatorname{\tau}}

\newlength{\hypobox}
\newlength{\gapbox}
\newcounter{hypo}
\renewcommand{\thehypo}{\textup{(H\arabic{hypo})}}

\newcounter{hypop}
\renewcommand{\thehypop}{\textup{(H\arabic{hypop}')}}

\newcounter{gap}
\renewcommand{\thegap}{\textup{(SG\arabic{gap})}}

\title{An optimal transportation approach to the decay of correlations for non-uniformly expanding maps}

\author{Beno\^{\i}t R. Kloeckner \thanks{Universit\'e Paris-Est, Laboratoire d'Analyse et de Mat\'ematiques Appliqu\'ees (UMR 8050), UPEM, UPEC, CNRS, F-94010, Cr\'eteil, France}}

\begin{document}


\maketitle

\begin{abstract}
We consider the transfer operators of non-uniformly expanding maps for potentials of various regularity, and show that a specific property of potentials (``flatness'') implies a Ruelle-Perron-Frobenius Theorem and a decay of the transfer operator of the same speed than entailed by the constant potential. The method relies neither on Markov partitions nor on inducing, but on functional analysis and duality, through the simplest principles of optimal transportation.

As an application, we notably show that for any map of the circle which is expanding outside an arbitrarily flat neutral point, the set of H\"older potentials exhibiting a spectral gap is dense in the uniform topology. The method applies in a variety of situation, including Pomeau-Manneville maps with regular enough potentials, or uniformly expanding maps of low regularity with their natural potential; we also recover in a united fashion variants of several previous results.
\end{abstract}

\section*{Foreword}

The published article (\textit{Ergodic Theory Dynam. Systems} 40, 2020) contained a significant error in Lemma 2.14, used in the core Theorem 4.1. The present text is a consolidated version of the article, with the error corrected (and a few other minor points improved along the way). In order to fix the error, the assumption on coupling in \ref{theo:core} needs to be slightly modified (see also Definition \ref{defi:stableDecay}) and Lemma 2.14 (now numbered \ref{lemm:unibound}) has been adjusted. Section \ref{ssec:stableCriterion} provides a criterion to ensure this new hypothesis in our cases of interest, so that all other results are unaffected. I apologize to readers of the previous version for this embarrassing mistake, and warmly thank Manuel Stadlbauer for pointing out this error to me.

\section{Introduction}

Chaotic dynamical systems are characterized by the unpredictability of individual orbits; their study has thus largely focused on the statistical behavior of typical orbits. The ergodic theorem, for example, is simply a law of large numbers for observables evaluated along an orbit. In many cases, it was realized that strong chaotic properties provide nice statistical properties just as much as they prevent the prediction of individual orbits. 

An important example of this phenomenon is provided by expanding maps $T:\Omega\to\Omega$ ($\Omega$ being e.g. a torus or a symbolic space). Close points being pulled apart by a factor, imprecision grows exponentially fast along an orbit; but if one wants to understand the statistical properties of an orbit, the order in which the points are considered does not matter. One can thus look at the orbit backward: each point $x_{t+1}$ is then followed by one of its inverse image $x_t \in T^{-1}(x_{t+1})$. This time reversal transforms the expanding property into a contraction property, but a choice between several inverse images is added. By making this choice randomly, one gets a Markov chain with a contraction property that can be used to prove nice statistical properties, which in turn can be translated for the original map. The \emph{thermodynamical formalism} can be seen as a general framework in which to deal with such arguments; Markov chains are extended to more general averaging operators acting on a space of observables, called \emph{transfer operators}. These operators $\op{L}_{T,A}$ are constructed using the map $T$ and a weight function $A$, called a potential, each (good enough) potential leading to an invariant measure $\mu_A$. This measure is called here a Ruelle-Perron-Frobenius (RPF) measure\footnote{The common name \emph{Gibbs measure} conflicts with other, related but different, notions.} because it is obtained by proving a Ruelle-Perron-Frobenius for the transfer operator. The game then consists in finding suitable hypotheses on the potential and observables ensuring fine properties of the transfer operator, which can be translated into the construction of $\mu_A$ and statistical properties for $\mu_A$-almost all orbits.

The most favorable case is when the transfer operator has a spectral gap: then one gets exponential decay of correlations, most usual limit theorems, analyticity of the ``pressure'', etc. More generally, one is interested in the decay of correlations for the measure $\mu_A$, which can be related to decay of the action of the transfer operator on observables of vanishing $\mu_A$-average.

As is long-known, a spectral gap occurs when the map is uniformly expanding and the potential and observables are H\"older continuous. The case of non-uniformly expanding maps or of less regular potential has then become the object of intense scrutiny. The main goal of the present article is to generalize a method first used in the uniformly expanding case in \cite{KLS}, to prove decay results for transfer operator for non-uniformly expanding maps and potentials and observables of various moduli of continuity (the most classical ones being $\alpha$-H\"older continuity). Decay of transfer operators in particular implies decay of correlations, but is stronger.

Compared to \cite{KLS}, our main contributions are to introduce a simple yet effective framework to deal with various rates of decay in a uniform fashion (Section \ref{ssec:rates}), to give a more central place to Markov chains, and to identify a simple condition on potentials, which we call flatness (Section \ref{ssec:flat}), under which the methods of \cite{KLS} can be expanded to prove a Ruelle-Perron-Frobenius theorem and decay properties.

We use an approach introduced by Denker and Urbansky \cite{DenkerUrbanski:1991b} to prove the Ruelle-Perron-Frobenius Theorem, which enables us to ``normalize'' the potential and to relate the dual operator $\op{L}_{T,A}^*$ to a Markov transition kernel. Then we use coupling arguments to prove a general convergence results for a weighted transition kernel in term of the convergence of the unweighted kernel. This is done using optimal transportation as a powerful tool to implicitly combine couplings. We obtain a spectral gap (and thus exponential decay of correlations) in many cases, but our methods also enable us to obtain polynomial decay in cases where a spectral gap is unlikely to occur. The same methods can be used to obtain intermediate or slower decay rates in suitable cases.

\subsection{Results with a spectral gap}

While the method is relatively general, we will largely focus in this introduction on an emblematic case, the Pomeau-Manneville family, and to similar maps. Given an exponent $q>0$, we consider on the circle $\mathbb{T}=\mathbb{R}/\mathbb{Z}$ parametrized by $[0,1)$ the map
\begin{align*}
T_q : \mathbb{T} &\to \mathbb{T} \\
  x &\mapsto \begin{cases} 
    \big(1+(2x)^q \big)x &\text{if }x\in [0,\frac12] \\ 
    2x-1 &\text{if } x\in [\frac12, 1)
             \end{cases}
\end{align*}

We denote by $\C^\alpha(\mathbb{T})$ the Banach space of $\alpha$-H\"older real functions, with the convention that $\C^1$ means Lipschitz rather than continuously differentiable. Our first result is a slight variant of a theorem of Li and Rivera-Letelier (see below).
\begin{theomain}[Spectral gaps in the Pomeau-Manneville family]\label{theo:mainHol}
If $q<1$, for any $\alpha, \gamma\in (0,1)$ such that $\gamma-\alpha\ge q$ we let $V= \C^\gamma(\mathbb{T})$;
if $q\ge 1$, for any $\gamma>q$ and $\alpha\in(0,1]$ such that $\gamma+1-\alpha\ge q$, we let $V$ be the linear subspace of $\C^\alpha(\mathbb{T})$ made of potentials $A$ that are continuously differentiable near $0$ and such that $A'(r)=O_{r\to0}(r^\gamma)$.

In both cases, for any potential $A\in V$ the transfer operator $\op{L}_{T_q,A}$ acting on $\C^{\alpha}(\mathbb{T})$ satisfies a Ruelle-Perron-Frobenius theorem with a spectral gap. In particular the RPF measure $\mu_A$ has exponential decay of correlations for all H\"older observables.\footnote{Of \emph{any} H\"older exponent: indeed as $\alpha$ decreases the space $\C^\alpha(\mathbb{T})$ becomes \emph{larger}; the loss in taking $\alpha$ small in not in the observables for which one gets exponential decay of correlation (on the contrary), but in the norm for which the transfer operator decays.}
\end{theomain}

Our method is different from the method of Li and Rivera-Letelier who proved this spectral gap (in the Keller space rather than the H\"older space) in the case $q<1$, with the same assumption $\gamma>q$ on the H\"older exponent  \cite{Li-RL, Li-RLb}.

\begin{rema}\label{rema:acim}
As Li and Rivera-Letelier observe, the bound $\gamma>q$ is optimal in the case $q<1$. The first reason is obvious: the potential $\log 1/\lvert T_q' \rvert$ is $q$-H\"older and its RPF measure, which is the absolutely continuous invariant measure (Acim), only has polynomial decay of correlations. Below we will obtain some result with lower rate of decay, but they will still not be applicable to Acims of intermittent maps: we always use a Ruelle-Perron-Frobenius Theorem on a space of continuous functions, while the density of the Acim is unbounded at $0$ in this case.
\end{rema}

The present method can be applied to other types of neutral points; for example when the graph of $T$ is less tangent to the diagonal at the neutral fixed point, the loss of regularity between the potential and the space its transfer operator acts upon is lower. To state a precise form of a sample result in this direction, let us introduce a particular family of Banach spaces refining the H\"older family. Given $\alpha\in[0,1)$ and $\beta\in\mathbb{R}$, let $\C^{\alpha+\beta\log}(\Omega)$ be the Banach space of functions $\Omega\to\mathbb{R}$ with modulus of continuity at most a multiple of
\[\omega_{\alpha+\beta\log}(r) = \frac{r^\alpha}{\big(\log\frac{r_0}r \big)^\beta}\]
(see Lemma \ref{lemm:holderlog} for the definition of $r_0$ and Section \ref{ssec:modulus} for the definition of the norm). When $\beta$ is positive, this imposes slightly more regularity than $\alpha$-H\"older, and when $\beta$ is negative slightly less. When $\alpha=0$, we impose $\beta>0$ and write simply $\C^{\beta\log}(\Omega)$ (which is a \emph{very loose} modulus of continuity).
\begin{theomain}\label{theo:mainless-tangent}
For some $q>0$, consider the map
\begin{align*}
T_{q\log} : \mathbb{T} &\to \mathbb{T} \\
  x &\mapsto \begin{cases} 0 & \text{if } x=0 \\
             \big( 1+ (1-\log 2x)^{-q}\big)x & \text{if } x\in(0,\frac12]\\
             2x-1 &\text{if }x\in[\frac12,1).
             \end{cases}     
\end{align*}
For any $\alpha\in (0,1)$ and any $A\in V := \C^{\alpha+q\log}(\mathbb{T})$, the transfer operator $\op{L}_{T_{q\log},A}$ acting on $\C^{\alpha}(\mathbb{T})$ satisfies a Ruelle-Perron-Frobenius theorem with a spectral gap (in particular as soon as $A$ is H\"older of any exponent, $\mu_A$ enjoys exponential decay of correlations for all H\"older observables).
\end{theomain}

For very neutral points (when the graph of $T$ is very tangent to the diagonal) one cannot merely adjust the regularity requirements without eliminating non-constant potentials; but one can instead ask for pointwise regularity at the neutral fixed point, as in the case $q\ge 1$ of Theorem \ref{theo:mainHol}. To deal with general neutral points, one can simply ask for the maximal possible regularity there.
\begin{theomain}[Density of spectral gap potentials]\label{theo:mainDense}
Let $T$ be a degree $k$ self-covering of the circle with a neutral fixed point $0$, uniformly expanding outside each neighborhood of $0$.
For any $\alpha\in (0,1]$, let $V$ be the linear space $V$ of $\C^\alpha$ potentials which are constant near the neutral point. Then for all $A\in V$, the transfer operator $\op{L}_{T,A}$ acting on $\C^{\alpha}(\mathbb{T})$ has a spectral gap.

For all $\gamma\in (0,\alpha)$, $V$ is dense in $\C^\alpha(\mathbb{T})$ for the $\gamma$-H\"older norm.
\end{theomain}
Note that the set of potentials the transfer operator of whose has a spectral gap is well-known to be open; in particular, this set is thus both open in the $\C^\alpha$ topology and dense in the uniform topology.
We do not need to impose any particular behavior at the neutral fixed point, which can be arbitrarily flat. We can for example apply this to examples of Holland \cite{holland2005slowly} (but not to the natural potential yileding the Acim, as Remark \ref{rema:acim} applies again).

Note that from the conclusions of each of Theorem \ref{theo:mainHol} or \ref{theo:mainless-tangent} we have, as in Theorem \ref{theo:mainDense}, that the set $\SG(T,\C^\alpha)$ of potentials whose transfer operator has a spectral gap contains an open set containing a linear subspace which is dense in the uniform norm.

As a consequence of Theorems \ref{theo:mainHol}-\ref{theo:mainDense} we can apply \cite{GKLM} under each of their sets of assumptions to easily obtain various results of classical flavor, e.g. providing formulas for successive derivatives of $\int \varphi \dd\mu_A$ with respect to $A$ and expressions for the modulus of convexity of the pressure function; and one can also argue as in \cite{KLS}  to show that the maximum entropy measure (or other RPF measures) depends on the map $T$ in a locally Lipschitz way, with respect to a Wasserstein metric $W_\alpha$ (see Corollary 1.5 of \cite{KLS}; the differentiability assumption is unnecessary in the context of e.g. Theorem \ref{theo:mainHol}, as one can see in the proof). We refer to these previous articles for these applications, but mention the following result as it needs some adaptation to an argument of \cite{GKLM}.
\begin{coromain}\label{coro:mainES}
In the context of Theorem \ref{theo:mainHol}, \ref{theo:mainless-tangent} or \ref{theo:mainDense}, for all $A\in V$ the RPF measure $\mu_A$  is the unique equilibrium state of the potential $A$, i.e. it uniquely maximizes
\[ h(\mu) + \int A \dd\mu\]
among $T$-invariant measures $\mu$, where $h$ is the Kolmogorov-Sinai entropy.
\end{coromain}


\subsection{Results with polynomial decay}

Our method applies readily to uniformly expanding maps and lower-regularity potentials, for example yielding the following.
\begin{theomain}[Polynomial decay in low regularity]\label{theo:mainUniform}
Let $T:\Omega\to\Omega$ be a $k$-to-$1$ uniformly expanding map of a compact metric space, and let $\beta\in(1,+\infty)$. For any potential $A\in\C^{\beta\log}(\Omega)$, the transfer operator $\op{L}_{T,A}$ acting on $\C^{(\beta-1)\log}(\Omega)$ satisfies a Ruelle-Perron-Frobenius theorem and the RPF measure has at least polynomial decay of correlations of degree $\beta-1$.
\end{theomain}
This can be generalized to other pairs of modulus of continuity (instead of $\omega_{\beta\log}$ for the potential and $\omega_{(\beta-1)\log}$ for the observables in the statement above).

Theorem \ref{theo:mainUniform} enables us to recover and slightly strengthen a result of Fan and Jiang \cite{FJ1,FJ2}.
\begin{coromain}[Acim for maps with mildly regular derivative]\label{coro:mainAcim}
Assume $\Omega$ is a connected manifold and $T:\Omega\to\Omega$ is continuously differentiable with $JT := \lvert \det DT\rvert:\Omega\to\mathbb{R}$ in $\C^{\beta\log}(\Omega)$ for some $\beta>1$.

Then $T$ has an absolutely continuous invariant measure in $\C^{(\beta-1)\log}(\Omega)$, with at least polynomial decay of correlations of degree $\beta-1$.
\end{coromain}
Note that compared to Corollary 1 in \cite{FJ2} we gain a logarithmic factor in the decay of correlations.
In fact (as in all our other results) we obtain more than decay of correlations: if $\op{L}=\op{L}_{T,-\log JT}$ is the transfer operator of the natural potential (whose RPF measure is the Acim of $T$), for all $f\in\C^{(\beta-1)\log}(\Omega)$ with zero average we get
\[\lVert \op{L}^t f \rVert_\infty \le \frac{C(f)}{t^{\beta-1}}\] 
In particular, if $\beta>3/2$, we have a polynomial decay of degree above $\frac12$ \emph{in the uniform norm}, and the Acim will satisfy a (Functional) Central Limit Theorem, see \cite{T-K05}.

We can also generalize Theorem \ref{theo:mainHol} as follows.
\begin{theomain}[Polynomial decay in the Pommeau-Manneville family]\label{theo:mainpol}
If $q<1$, for any $\beta>1$ and any $A\in\C^{q+\beta\log}(\mathbb{T})$, the transfer operator $\op{L}_{T_q,A}$ acting on $\C^{(\beta-1)\log}(\mathbb{T})$ satisfies a Ruelle-Perron-Frobenius theorem, and the RPF measure has a polynomial decay of correlations of degree $\beta-1$.
\end{theomain}
We could also obtain with the same method a similar result for $q\ge1$.

Note that the $\C^{(\beta-1)\log}$ regularity is very mild, which is a strength of Theorems \ref{theo:mainUniform} and \ref{theo:mainpol} as they apply to many observables, much less regular than H\"older. Note that for example, in the case of a Pommeau-Manneville map $T_q$ with $q<1$, when $A$ is $\gamma$-H\"older for some $\gamma>q$ Theorem \ref{theo:mainHol} is not applicable to observables in $\C^{\beta\log}$. But then one can use Theorem \ref{theo:mainpol} since in particular $A\in \C^{q+(\beta+1)\log}$, and then one obtains polynomial decay of correlations for such very weakly regular observables.


\subsection{Short account of some previous works}

The works on decay of correlations for non-uniformly expanding maps are too numerous to all be cited; let us only mention a few of them in addition the the ones already discussed above. Manneville and Pomeau \cite{PM} introduced the $q=1$ case of the family $(T_q)_{q>0}$ as a model for intermittent phenomena observed in the Lorentz model; the thermodynamical formalism has been studied for such intermittent maps at least since the work of Prellberg and Slawny \cite{PS}. Liverani, Saussol and Vaienti \cite{LSV} obtained good estimates of the decay of correlation for the Acims of the Pomeau-Manneville family with a simple approach. Young introduced the now called ``Young towers'' \cite{Young98,Young99}, giving an upper bound on decay of correlation for a wealth of non-uniformly hyperbolic systems. Her method was for example used by Holland \cite{holland2005slowly} to study maps with strongly neutral point, proving in some cases sub-polynomial decay of correlations. Fine statistical properties have notably been established by Gou\"ezel \cite{Gouezel05} and Melbourne and Nicol \cite{MN}. Hu \cite{Hu04} and Sarig \cite{Sarig02} proved \emph{lower} bounds on decay of correlations for intermittent maps, refined by Gou\"ezel \cite{gouezel2004sharp}. 

All the above works deals with the Acim, and their results are therefore formally disjoint from our results \ref{theo:mainHol}-\ref{coro:mainES}, \ref{theo:mainpol}.
It is often argued that the Acim or, more generally, the physical measures, are the invariant measures that matter most since Lebesgue-negligible events seem too elusive to be ever witnessed. However, we would like to stress that Guih\'eneuf (Theorem 56 in \cite{guiheneuf}) showed that in some cases (for generic conservative homeomorphisms), \emph{all} invariant measures will in fact be ``observed'' in numerical simulations. The fact that RPF measures are in many cases equilibrium states for their potential is another reason to study them in general. 

Equilibrium states of intermittent maps have notably been studied by Hofbauer and Keller \cite{Hofbauer-Keller} and Bruin and Todd \cite{Bruin-Todd}, where they deal with potential that are sufficiently close to being constant. By contrast, our results need no such assumption but rely instead on regularity hypotheses. Note that, as will be seen below and is visible in Theorem \ref{theo:mainHol} when $q\ge 1$ and in Theorem \ref{theo:mainDense}, we in fact mostly need regularity of the potential near the neutral point; away from the neutral point, the potential can be merely H\"older (and observables are allowed to be arbitrary in the suitable regularity class, without need for a special behavior near the neutral point).

Previously to the aforementioned work by Li and Rivera-Letelier, uniqueness of equilibrium states and exponential decay of correlations where obtained in some cases by Liverani, Saussol and Vaienti \cite{LSV98}, using the Hilbert metric on cones. These authors do get a decay of the transfer operator, but in uniform norm with the BV norm of the observable as a factor -- this inhomogeneity makes their result intermediate between ``naked'' exponential decay of correlations and a spectral gap. More recently, Castro and Varandas \cite{CV} obtained interesting results for a large family of non-uniformly expanding map, but they need the potential to be very close to a constant in H\"older norm (in this direction, see also \cite{K:HT}). Compared to these works, the main features of the present approach thus are: to allow potentials with large variations; to provide spectral gaps in many cases; to be also applicable in situation where decay of correlations are likely not to be more than polynomial. We also expect the method to be flexible enough to be used in a wide array of examples; in any case, the method feels sufficiently different from the ones currently available (such as inducing) to be potentially useful beyond our main results.

Cyr and Sarig \cite{cyr-Sarig09} proved that the spectral gap property for the transfer operator is dense for countable Markov shifts, but in a sense that enables any rich enough Banach space. This differs quite a bit from Theorem \ref{theo:mainDense} where we get uniform density of the spectral gap property for transfer operators acting on the fixed, natural $\C^\alpha$ Banach space.

The H\"older moduli of continuity are quite ubiquitous in the literature, and we finish by mentioning the work of Lynch \cite{Lynch}, who used Young towers to study the decay of correlations for the Acim of (possibly non-uniformly) expanding maps and observables of various weak regularities, including $\C^{\beta\log}$. In the uniformly expanding case, for $\C^{\beta\log}$ observables he obtained polynomial decay of correlation but only with a loosely controlled degree, and Theorem \ref{theo:mainUniform} is much more precise in this case. Our method could be adjusted to work with other modulus of continuity as well.

\subsection{Structure of the article}

In order to both obtain clean and easily stated results and make our method easily applicable in other situations, this article is constructed in layers. Depending on the case one wants to apply our method to, one may use results stated in a high layer, or may have to start from one of the first layers and specialize it to one's precise case.

The ``zeroth'' layer, Section \ref{sec:preliminary}, sets up notation and definitions of general scope: optimal transportation, moduli of continuity and associated generalized H\"older spaces, general considerations on convergence speed to the fixed point for iterations of weakly contracting maps, couplings of a transition kernel, flatness of a potential. We also explain how Corollary \ref{coro:mainES} follows from Theorems \ref{theo:mainHol}-\ref{theo:mainDense} and \cite{GKLM}.

The first layer is centered on transition kernels: in Section \ref{sec:positive} we prove the Ruelle-Perron-Frobenius, and in Section \ref{sec:maincontraction} we prove our core contraction result, Theorem \ref{theo:core}.

The second layer, Section \ref{sec:chains}, specializes this contraction result to transition kernels that are backward random walk for $k$-to-$1$ maps. Several Lemmas aiming at proving flatness for various potential in various contexts are proved, and the first layer is condensed into Theorem \ref{theo:nuem} for $k$-to-$1$ maps and $\C^{\alpha+\beta\log}$ potentials. The previous Section should be useful in more generality, see Remark \ref{rema:nopartition}

The third and last layer, Section \ref{sec:proofmain}, relies on the previous one to finish the proofs of the results stated in this introduction.

\section{General setting and preliminary results}
\label{sec:preliminary}

In all the article $\Omega$ denotes a compact metric space, with distance function $d$, and $T:\Omega\to\Omega$ is a map. We denote by $\proba(\Omega)$ the set of probability measures on $\Omega$, endowed with the weak-* topology. Given $\mu\in\proba(\Omega)$ and a Borel-measurable $f:\Omega\to\mathbb{R}$, we denote the integral of $f$ with respect to $\mu$ either by $\int_{\Omega} f \dd\mu$ or $\mu(f)$.

If $x,y$ are quantities depending on some parameters, when writing $x\le C y$ we may dispense from introducing the constant $C$, which may change from paragraph to paragraph. We will only introduce $C$ more explicitly when we feel there is a risk of confusion, notably on the dependance of $C$ on some of the parameters. Sometimes we will indicate a change in the constant more explicitly, for example writing $x_2 \le C' y_2$. When we prefer to let the constant $C$ completely implicit, we write $x\lesssim y$.

We denote by $\one$ the constant function on $\Omega$ with value $1$.

\subsection{Maps, transfer operators and transition kernels}
\label{ssec:transferop}

 Following the ideas of thermodynamical formalism, to construct and study invariant measures one considers a Banach algebra $\fspace{X}$ (i.e. $\fspace{X}$ is stable by product and $\lVert fg\rVert\le \lVert f\rVert \lVert g\rVert$)  of ``potentials'' $A:\Omega\to\mathbb{R}$ and a Banach space $\fspace{Y}$ of ``observables'' $A:\Omega\to\mathbb{R}$ such that for any potential the following ``transfer operator'' is bounded on $\fspace{Y}$:
\begin{equation}
\op{L}_{T,A} f(x) = \sum_{T(y)=x} e^{A(y)} f(y) \qquad\forall f\in\fspace{Y}
\label{eq:transfer-classical}
\end{equation}
(we shall let either or both subscripts implicit when $T$ or, more rarely $A$, is clear from the context). The hypothesis that $\fspace{X}$ is an algebra is only meant to ensure that $e^A \in\fspace{X}$ so that we speak indifferently of the regularity of $A$ or $e^A$; and in most cases $\fspace{X}$ will be equal to or a subspace of $\fspace{Y}$. 

The above formula is suitable when $T$ is $k$-to-one, but poses problems otherwise. It could be ill-defined if the number of inverse images is infinite, or could map continuous functions to non-continuous ones if the number of inverse images is not locally constant. These issues can often be dealt with, but solution sometimes feel ad-hoc. It is thus natural to replace the sum with an integral with respect to a Markov transition kernel (all measurability  properties shall be considered with respect to the Borel algebra).
\begin{defi}
By a \emph{transition kernel} on $\Omega$ we mean a family $\mchain{M}=(m_x)_{x\in\Omega}$ of probability measures $m_x$ on $\Omega$ (we ask $x\mapsto m_x$ to be furthermore Borel measurable).

The transition kernel $\mchain{M}$ is said to be a \emph{backward walk} of the map $T$ if for all\footnote{This could be replaced by ``almost all $x$'' with respect to the RPF measure for most purposes, the problem being that at this point the RPF measure is not known.}  $x\in\Omega$, the measure $m_x$ is supported on $T^{-1}(x)$.
Given a potential $A\in\fspace{X}$, we define the \emph{transfer operator} of $\mchain{M}$ with respect to $A$ by
 \[\op{L}_{\mchain{M},A} f(x) = \int_\Omega e^{A(y)} f(y) \dd m_x(y)\]
(as above we may keep either or both subscripts implicit.)
A transition kernel is said to be \emph{compatible} with $\fspace{X}$ if for all $A\in\fspace{X}$, the above formula defines a continuous operator $\fspace{X}\to\fspace{X}$
\end{defi}

As soon as $\fspace{X}$ is large enough to separate probability measures, i.e. 
\[\big( \forall f\in\fspace{X}: \mu(f)=\nu(f) \big)\implies \mu=\nu,\]
one can define by duality the operator $\op{L}_{\mchain{M},A}^*$ acting on $\proba(\Omega)$:
\[ \int_\Omega f \dd \big(\op{L}_{\mchain{M},A}^* \mu\big) = \int_\Omega \op{L}_{\mchain{M},A} f \dd\mu \qquad \forall f\in\fspace{X}.\]
Equivalently, 
\[\dd\big(\op{L}_{\mchain{M},A}^* \mu \big) (x)= \int_\Omega (e^{A} \dd m_x) \dd\mu(x).\]

A classical computation shows that $\op{L}_A^t$ can be expressed as
\[\op{L}_A^t f(x) = \int e^{A^t(\bar x)} f(x_t) \dd m^t_x(\bar x)\]
where $A^t$ denote the Birkhoff sum:
\begin{align*}
A^t: \qquad\qquad\qquad \Omega^t &\to \mathbb{R} \\
  \bar x=(x_1,\dots,x_t) &\mapsto A(x_1)+\dots + A(x_t).
\end{align*}

\begin{exem}
The most classical example is when $T$ is a $k$-to-one local homeomorphism, $m_x$ is the uniform distribution over $T^{-1}(x)$, and $\fspace{X}$ is the algebra of continuous, or $\alpha$-Hölder functions. Then the transfer operator of $\mchain{M}$ is equal to the classical transfer operator of $T$ up to a constant. When $T$ is an unimodal map $[0,1]\to[0,1]$ with critical point $c$ and such $T$ maps each of $[0,c]$ and $[c,1]$ onto $[0,1]$, since $x=1$ has only one inverse image while the $x<1$ have two the formula \eqref{eq:transfer-classical} is somewhat inappropriate: the transfer operator would not preserve the space of continuous functions. Taking $m_x$ to be the uniform distribution over $T^{-1}(x)$ solves this inconvenience.

Another classical example is with $T$ a finite-to-one, piecewise continuous map, $m_x$ the uniform distribution over $T^{-1}(x)$, and $\fspace{X}$ the algebra of BV functions.
\end{exem}

While we will be primarily interested in transition kernels which are backward walks for maps with expanding properties, the question of the spectral gap for $\op{L}_{\mchain{M},A}$ is relevant in all generality.

It is common to single out the potentials $A$ such that $\op{L}_A\one=\one$; these potentials are said to be \emph{normalized}, and can be characterized in several manners. One particularly relevant one is to observe that $A$ is normalized exactly when $e^A \dd m_x$ is a probability measure for all $x$, i.e. when  $(e^A\dd m_x)_{x\in\Omega}$ is a transition kernel.

Let $\fspace{X}$ be a Banach space of functions defined on $\Omega$, whose norm is denoted by $\lVert\cdot\rVert$, and assume that $\op{L}_{\mchain{M},A}$ acts continuously on $\fspace{X}$ \begin{defi}
We shall say that $\op{L}_{\mchain{M},A}$ satisfies a Ruelle-Perron-Frobenius (RPF) Theorem on $\fspace{X}$ if there exist a positive function $h\in\fspace{X}$ and a positive constant $\rho$ such that $\op{L}_{\mchain{M},A} h=\rho h$, and there exist a positive, finite measure $\nu$ such that $\op{L}_{\mchain{M},A}^*\nu= \rho \nu$. 

Then the positive measure $\mu_A$ defined by $\dd\mu_A = h \dd\nu$ (choosing the eigenfunction $h$ so as to make it a probability)  is called the \emph{RPF measure} of the potential $A$.

We will say that $\op{L}_A$ has a spectral gap if it satisfies the RPF Theorem and for some $C\ge 1$ and $\delta\in(0,1)$ and all $t\in\mathbb{N}$, all $f\in \fspace{X}$ such that $\nu(f)=0$ we have
\[\lVert \op{L}_{\mchain{M},A}^t f \rVert \le C \rho^t (1-\delta)^t \lVert f\rVert\]
\end{defi}
The definition of spectral gap above may seem formally stronger than more usual definitions, but we will get it in this form and it is the definition needed to apply \cite{GKLM}; moreover it can be shown to be equivalent to more standard definitions, see e.g. \cite{K:HT}.\\

A classical computation shows that whenever $\fspace{X}$ is large enough and $\mchain{M}$ is a backward walk for a map $T$, the RPF measure $\mu_A$ is $T$-invariant. Since the framework of transition kernels is not completely standard, let us give some details.

First one observes that for all $f\in\fspace{X}$, we have
\begin{align*}
\op{L}_{\mchain{M},A}(g\cdot f\circ T) (x)
  &= \int_\Omega e^{A(x_1)} g(x_1) f(\underbrace{T(x_1)}_{=x}) \dd m_x(x_1) = f(x) \cdot \op{L}_{\mchain{M},A}(g) (x)\\
\end{align*}

Then setting 
\[\tilde A = A + \log h_A-\log h_A\circ T -\log \rho_A\]
we get a new potential which is normalized and such that $\op{L}_{\mchain{M},A}$ and $\op{L}_{\mchain{M},\tilde A}$ are conjugated up to a constant. Indeed:
\[\op{L}_{\mchain{M},\tilde A}f(x) = \int e^{A(x_1)}\frac{h(x_1)}{\rho h(T(x_1))} f(x_1) \dd m_x(x_1) = \frac{1}{\rho h} \op{L}_{\mchain{M},A}( hf ) (x)\]
and in particular
\[\op{L}_{\mchain{M},\tilde A} \one = \frac{1}{\rho h} \op{L}_{\mchain{M},A} h = \one\]
i.e. $\tilde A$ is normalized, which can also be interpreted as having $1$ as eigenvalue and $\one$ as eigenfunction. Similarly, one shows that $\op{L}_{\mchain{M},A}^*$ has eigenprobability $\mu_A$. Then for all $f\in\fspace{X}$:
\[\int f\circ T \dd\mu_A = \int f\circ T \dd\big(\op{L}_{\mchain{M},\tilde A} \mu_A\big) = \int \op{L}_{\mchain{M},\tilde A}(\one \cdot f\circ T) \dd\mu_A = \int f \dd\mu_A\]
As soon as $\fspace{X}$ separates measures, we can deduce that $\mu_A$ is $T$-invariant. To proceed as above, we have implicitly assumed that $\op{L}_{\mchain{M},\tilde A}$ also acts on $\fspace{X}$; this and the separation property will be easily seen to hold in all our cases of interest.

\subsection{Generalized H\"older spaces}\label{ssec:modulus}

In this Section we describe a class of Banach algebras of functions providing a good compromise between generality and simplicity.

It has become customary in dynamical systems to use H\"older spaces, which are nothing else than spaces of functions having a modulus of continuity of power type. It makes about as much sense to consider arbitrary modulus of continuity, and this leads to generalized H\"older spaces.

We will call \emph{modulus of continuity} a continuous, increasing, concave function $\omega:[0,+\infty) \to[0,+\infty)$ such that $\omega(0)=0$, and say that a function $f:\Omega\to\mathbb{R}$ (or $\mathbb{C}$) is $\omega$-H\"older if for some constant $C$ we have
\[ \lvert f(x)-f(y) \rvert \le C\omega(d(x,y)) \qquad \forall x,y\in \Omega.\]
The least such $C$ is denoted by $\Hol_\omega(f)$ and we denote by $\C^\omega(\Omega)$ the space of all such functions, endowed with the norm
\[\lVert f\rVert_\omega := \lVert f\rVert_\infty + \Hol_\omega(f).\]
This makes $\C^\omega(\Omega)$ a Banach algebra, i.e. it is complete and $\lVert fg \rVert_\omega \le \lVert f\rVert_\omega \lVert g\rVert_\omega$. This is elementary and can be seen directly, or it can be deduced from the analogous result for Lipschitz functions (i.e. the case $\omega(t)=t$): indeed $\omega$-H\"older functions are nothing else than Lipschitz functions with respect to the metric $\omega\circ d$ obtained by composing the original metric $d$ with the modulus $\omega$  (this is indeed a metric since $\omega$ is concave).

This observation could make one think that we did not gain any generality; but in many cases the modified distance $\omega\circ d$ is less natural than $d$ and, more importantly, we will rely on expansion properties of the map $T:\Omega\to\Omega$, which will be expressed in the metric $d$; in general these properties are not invariant under such a change of distance. Moreover we will sometimes ask potentials and observable to have different regularities, and will thus consider several moduli of continuity simultaneously.

Note that, $\Omega$ having finite diameter, only the germ at zero of $\omega$ truly matters: changing $\omega$ away from zero only changes $\lVert\cdot\rVert_\omega$ into an equivalent norm.\\

The H\"older moduli of continuity, $\omega_\alpha(t) = t^\alpha$ for $\alpha\in (0,1]$, are of particular importance. In subscripts, we shall replace $\omega_\alpha$ by $\alpha$.
We will use a generalization including powers of logarithm, which defined by the following.
\begin{lemm}\label{lemm:holderlog}
For all $\alpha\in (0,1]$ and $\beta\in \mathbb{R}$, there is a modulus of continuity $\omega_{\alpha+\beta\log}$ such that
\[\omega_{\alpha+\beta\log}(r)\sim \frac{r^\alpha}{\lvert \log r\rvert^{\beta}} \text{ as }r\to 0\]
and such that for some $\theta\in (0,1)$ and all $r\in [0,1]$
\[\omega(r/2) \le \theta \omega(r).\]
\end{lemm}

\begin{proof}
For $r\in (0,1]$ we write 
\[\omega_{\alpha+\beta\log}(r) := \frac{r^\alpha}{(\log \frac{r_0}{r} )^\beta}\]
where $r_0$ is to be determined. Note that for each $r_0$, the asymptotic behavior at $0$ is as desired.
If $r_0$ is large enough,
the above formula defines an increasing and concave function on $[0,1]$, that can thus be extended into a modulus of continuity.

If $\beta\ge 0$, we have $\omega(r/2) \le 2^{-\alpha}\omega(r)$ so that we can take $\theta=2^{-\alpha}$, and we turn to the case $\beta <0$. Then 
\[\omega(r/2) \le \theta\omega(r) \quad\text{with}\quad
\theta = 2^{-\alpha}\Big(1+\frac{\log 2}{\log r_0} \Big)^{-\beta}\]
and it suffices to take $r_0$ large enough to ensure $\theta<1$.
\end{proof}

When $\alpha=0$, we impose $\beta>0$ and take
\[\omega_{\beta\log}(r) := \big(\log\frac{r_0}r\big)^{-\beta} \qquad\forall r\in [0,1]\]
where $r_0$ is large enough (in function of $\beta$) to ensure $\omega_{\beta\log}$ is increasing and concave on $[0,1]$. Then we extend it arbitrarily (e.g. linearly) to $[0,+\infty)$. Lemma \ref{lemm:holderlog} does not stand for $\alpha=0$.

\subsection{Proof of Corollary \ref{coro:mainES}}

Corollary \ref{coro:mainES} states that, in the context of any of Theorems \ref{theo:mainHol}-\ref{theo:mainDense}, the RPF measure of a potential $A\in V$ (obtained as $\dd\mu_A=h_A\dd\nu_A$ where $h_A$ is positive eigenfunction of $\op{L}_{T,A}$ an $\nu_A$ is a probability eigenmeasure of the dual operator) is the unique equilibrium state of $A$. This is of a very classical flavor, but given the diversity of assumptions one finds in the literature it deserves a proof in our precise situation. 

We use \cite{GKLM}, Section 7, but some adaptation is needed: we don't have a spectral gap for \emph{all} potentials in $\C^\alpha(\mathbb{T})$. However we have a spectral gap for an open set $\mathcal{U}$ containing the linear subspace $V$, which is $\lVert\cdot\rVert_\infty$-dense in $\C^\alpha(\mathbb{T})$, and this will be sufficient.

First, observe that for all $A\in\mathcal{U}$ we have $P(A)=\log \lambda_A$, where $P$ is the topological pressure and $\lambda_A$ is the leading eigenvalue of $\op{L}_{T,A}$. To see this, one only has to prove that $P(A)=0$ whenever $A$ is normalized: indeed each potential in $\mathcal{U}$ differs from a normalized potential by a coboundary $\log h_A-\log h_A\circ T$ and a constant $\log\lambda_A$, and both $\lambda_A$ and $P(A)$ change in the same way when $A$ is added a coboundary and a constant. To see that $P(A)=0$ whenever $A$ is normalized, one considers the two definitions of pressure based on $(n,\varepsilon)$-separated sets and $(n,\varepsilon)$-spanning sets respectively, and applies them to a set obtained as follows. Let $E_0$ be a set of $1/\varepsilon$ points regularly spaced on the circle, then let $E$ be the set of all inverse images by $T^n$ of elements of $E_0$. Since $T$ is non-contracting, $E$ is both $(n,\varepsilon)$-separated and $(n,\varepsilon)$-spanning. The normalization property enables one to group the terms of $\sum_{x\in E} e^{\sum_{i=0}^{n-1} A(T^i(x))}$ into $1/\varepsilon$ sums over each set of inverse images of each element $x_0\in E_0$, which by normalization each sum to $1$. It follows that $P(A)=0$ whenever $A$ is normalized, and thus $P(A)=\log\lambda_A$.

By Theorem 9.12 in \cite{Walters} (note that $h$ is upper semi-continuous since $T$ is expansive), one can write 
\[h(\mu)= \inf_{\varphi\in \C^0(\mathbb{T})} \Big\{ P(\varphi)-\int \varphi \dd\mu \Big\}.\]
Since $V$ is $\lVert\cdot\rVert_\infty$-dense in $\C^\alpha(\mathbb{T})$, and therefore dense in $\C^0(\mathbb{T})$, and since $P$ is continuous in the uniform norm (\cite{Walters} Theorem 9.7) we can rewrite this as
\[h(\mu)= \inf_{B\in V} \Big\{ P(B)-\int B \dd\mu \Big\} = \inf_{B\in V} \Big\{ \log\lambda_B-\int B \dd\mu \Big\},\]
or yet $h(\mu)=\inf_{V} H(\mu,B)$ where $H(\mu,B):= \log\lambda_B-\int B \dd\mu$.

We are trying to maximize the functional $P_A(\mu):= h(\mu)+\int A \dd\mu$, which amounts to maximizing
\[P_A(\mu)-P(A) = \inf_{B\in V} H(\mu,B)+\int A \dd\mu -\log\lambda_A = \inf_{B\in V} \big( H(\mu,B) - H(\mu,A) \big)\]
For $\mu\neq\mu_A$, the formula for the derivative of $H(\mu,B)$ with respect to $B$ (\cite{GKLM} Proposition 7.2) and the density of $V$ show that there is a $B\in V$ near $A$ such that $H(\mu,B)<H(\mu,A)$, so that $P_A(\mu)<P(A)$. We only have left to prove that $H(\mu_A,B)\ge H(\mu_A,A)$ for all $B\in V$; but  $A$ is a critical point of $B\mapsto H(\mu_A,B)$ which is a convex functional on $V$, and we are done.

\subsection{Rates of decay}\label{ssec:rates}

In this section, we establish some notation and a few elementary  results to study the rate of decay of the iteration of a map. Let $X$ be a metric space, whose metric shall be denoted by $\rho$; in the sequel, $X$ will be either the phase space $\Omega$, or a Banach space of potential or observables, or a space of measures.

Let $P:X\to X$ be a map (which will be either an inverse branch of the dynamical system under study, or a transfer operator, or the dual of a transfer operator). To motivate the next definition, consider 
\[F(t,r) := \sup_{\substack{\rho(x,y)\le r \\ n\ge t}} \rho(P^n(x),P^n(y)) \]
and observe that
$F$ is non-increasing in $t$, non-decreasing in $r$, and satisfies
\[ F(t_1+t_2,r) \le F(t_1,F(t_2,r)).\]
This property encompasses a lot of information, as short time information at some scales reflects on longer time information at some other scales. This is what we shall harness to get effective bounds.
\begin{defi}
A non-negative function $F : \mathbb{N} \times (0,R) \to (0,+\infty)$ (where $R$ is a positive number or $+\infty$) is said to be a \emph{decay function} if
\begin{enumerate}
\item $F(t,r)$ is non-increasing in $t$, non-decreasing and concave in $r$,
\item $F(t,r)\to 0$ as either $t\to\infty$ or $r\to 0$, the other variable being fixed,
\item for some $C>0$ and all $t,r$: $F(t,r)\le Cr$,
\item for all $t_1,t_2$ and $r$: $F(t_1+t_2,r) \le F(t_1,F(t_2,r))$.
\end{enumerate}
\end{defi}
The concavity in $r$ will be important when using optimal transport, as it will enable us to bound above an integral of the decay function by the decay function of an integral. The third condition, corresponding to a uniform Lipschitz condition on the maps $(P^t)_{t\in\mathbb{N}}$, will ensure some uniformity of the behavior of $F$ (it is used implicitly in Lemmas \ref{lemm:expdec} and \ref{lemm:poldec} below).

We shall say that a map $P$ has decay rate $F$ (meaning implicitly: ``at least $F$''; sometimes we specify ``in the metric $\rho$'') if for all $t,x,y$,
\begin{equation}
\rho(P^t(x),P^t(y)) \le F(t,\rho(x,y)).
\label{eq:decay}
\end{equation}

It will be convenient to introduce for all $\theta\in (0,1)$
the \emph{decay times} of $F$ as
\[ \hl_\theta(r) = \min \{t\in\mathbb{N} : F(t,r)\le \theta r \} = \min \{t\in\mathbb{N} :  \forall s\ge r, \ F(t,s)\le \theta s \} \]
Where the second equality comes from concavity of $F(t,\cdot)$.

The standard notion of half-life corresponds to $\tau_{\sfrac{1}{2}}$, and is constant in case of exponential decay. More generally, we 
have the following elementary result.
\begin{lemm}\label{lemm:expdec}
The following are equivalent:
\begin{enumerate}
\item\label{enumi:expdec1} there exist $C\ge 1$, $\delta\in(0,1)$ such that $F(t,r)\le C (1-\delta)^t r$ for all $t,r$,
\item\label{enumi:expdec2} for some $\theta\in (0,1)$, there exist $D>0$ such that for all $r$: $\hl_\theta(r)\le D$,
\item\label{enumi:expdec3} for all $\theta\in (0,1)$, there exist $D>0$ such that for all $r$: $\hl_\theta(r)\le D$.
\end{enumerate}
\end{lemm}
When $P$ is Lipschitz, Lemma \ref{lemm:expdec} provides a (uniform) exponential decay for $P$ as soon as for some $\theta\in(0,1)$, some $t_0\in\mathbb{N}$ and all $x,y\in X$:
\[\rho(P^{t_0}(x),P^{t_0}(y)) \le \theta \rho(x,y).\]

Concerning polynomial decay, we have the following analogue.
\begin{lemm}\label{lemm:poldec}
Let $\alpha$ be a positive real number. The following are equivalent:
\begin{enumerate}
\item\label{enumi:poldec1} there exist $B\ge 1,b\in(0,1)$ such that $F(t,r)\le \frac{Br}{(t r^\alpha+b)^{\frac{1}{\alpha}}}$ for all $t,r$,
\item\label{enumi:poldec2} for some $\theta\in (0,1)$, there exist $D>0$ such that $\hl_\theta(r)\le \frac{D}{r^\alpha}$ for all $r$,
\item\label{enumi:poldec3} for all $\theta\in (0,1)$, there exist $D>0$ such that $\hl_\theta(r)\le \frac{D}{r^\alpha}$ for all $r$.
\end{enumerate}
\end{lemm}
When these conditions are realized, we will say that $F$ is polynomial (with degree $1/\alpha$).

\begin{proof}
It is clear that \ref{enumi:poldec3} implies \ref{enumi:poldec2} and that \ref{enumi:poldec1} implies \ref{enumi:poldec3}. 

Assume \ref{enumi:poldec2} and let us prove \ref{enumi:poldec1}. Let $t\in\mathbb{N}$, $r\in [0,\infty)$. Let $k\in\mathbb{N}$ be the largest integer such that
\[ t \ge \frac{D}{r^\alpha}  + \frac{D}{ \theta^\alpha r^\alpha}  + \dots + \frac{D}{\theta^{\alpha k}r^\alpha} \]
(taking $B$ large enough, $b$ small enough we can easily deal with the range $t\le D/ r^\alpha$, which we thus ignore from now on).
We have $F(t,r) \le \theta^k r$ and:
\begin{align*}
t &\le \frac{D}{r^\alpha} + \frac{D}{ \theta^\alpha r^\alpha}  + \dots + \frac{D}{\theta^{\alpha (k+1)}r^\alpha}  \\
t &\le \frac{D}{r^\alpha} \frac{\theta^{-\alpha(k+2)} -1}{\theta^{-\alpha}-1} \\
\theta^{-\alpha(k+2)} &\ge tr^\alpha \frac{\theta^{-\alpha}-1}{D} +1 \\
\theta^k &\le \frac{1}{\big( tr^\alpha (\theta^\alpha-\theta^{2\alpha})/D+1\big)^{\frac{1}{\alpha}}}\\
F(t,r) &\le \frac{Br}{(tr^\alpha+b)^{\frac{1}{\alpha}}} \qquad\text{for some }B,b.
\end{align*}
\end{proof}

We shall only consider the two families of decay functions given in the first items of Lemmas \ref{lemm:expdec} and \ref{lemm:poldec}, but more general decay rates and decay times can be considered.

\subsection{Optimal transportation}\label{ssec:transport}

Optimal transportation is a vast subject, from which we will only use the Wasserstein metric with respect to various distances of the form $\omega\circ d$ where $\omega$ is a modulus of continuity. One defines it for any $\mu,\nu\in\proba(\Omega)$ as the least cost needed to move $\mu$ to $\nu$, when the cost of moving an amount of mass from $x$ to $y$ is $\omega\circ d(x,y)$. Formally:
\[\wass_\omega(\mu,\nu) = \inf_{\Pi\in\Gamma(\mu,\nu)} \int_{\Omega\times\Omega} \omega\circ d(x,y) \dd\Pi(x,y)\]
where $\Gamma(\mu,\nu)$ is the set of probability measures on $\Omega\times\Omega$ with marginals $\mu$ and $\nu$; one calls elements of $\Gamma(\mu,\nu)$ \emph{transport plans}, and a transport plan is said to be optimal (with respect to $\omega$) if it realizes the infimum in $\wass_\omega(\mu,\nu)$. Optimal transport plans always exist, and $\wass_\omega$ metrizes the weak-$*$ convergence of measures (here the compactness of $\Omega$ is used). The space $\proba(\Omega)$ endowed with $\wass_\omega$ is compact (and in particular complete).

The Wasserstein metric has the advantage of being flexible: from a upper bound, we can deduce the existence of a transport plan with small cost, and to obtain a upper bound we only have to produce a transport plan with small cost. It makes it quite easy to prove that good transport plans exist by a sequence of small arguments (notably using the fact that $\wass_\omega$ is indeed a metric: it satisfies the triangular inequality) that would otherwise need careful explicit coupling arguments.

To prove lower bounds or to use upper bounds on the Wasserstein metric, an important tool is the Kantorovich duality ensuring that we can rewrite the Wasserstein metric as follows:
\[\wass_\omega(\mu,\nu) = \sup_{\Hol_\omega(f)\le 1} \lvert  \mu(f)- \nu(f)\rvert.\]
The inequality $\ge$ is the easy one, obtained by a direct computation, and is in fact the only one we shall use.

The following generalization of Corollary 5.2 from \cite{KLS}  is simple but central to our methods.
Let $\op{L}$ be a positive bounded operator on $\C^\omega(\Omega)$ such that $\op{L}\one=\one$ and let $\op{L}^*$ be its dual, acting on 
probability measures of $\Omega$. Let $F$ be a decay function and let $\omega'$ be a second modulus of continuity.

\begin{prop}\label{prop:WassDecay}
If $\op{L}^*$ has decay rate at least $F$ in the distance $\wass_\omega$ (recall \eqref{eq:decay}), then it has a unique invariant probability measure $\mu$ and for all $f\in\C^\omega(\Omega)$  we have
\[ \lVert \op{L}^t f - \mu(f) \rVert_\infty \le \Hol_\omega(f) F(t,\omega(\diam \Omega)) \quad\text{and}\quad \Hol_{\omega'}(\op{L}^t f) \le \Hol_{\omega}(f) F^\omega_{\omega'}(t)\]
where 
\[F^\omega_{\omega'}(t) := \sup_{s\in (0,\diam\Omega)} \frac{F(t,\omega(s))}{\omega'(s)}.\]
In particular, when $F(t,r)=C(1-\delta)^t r$ (where $C\ge 1$ and $\delta\in(0,1)$), we have
\[\lVert \op{L}^t f -\mu(f) \rVert_\omega \le C'(1-\delta)^t \Hol_\omega(f)\]
\end{prop}

\begin{proof}
The existence of an invariant measure is standard, and uniqueness follows from the decay hypothesis.

For all $x\in\Omega$ we can write $\op{L}^t f(x) = \int \op{L}^t f \dd \delta_x = \int f \dd \op{L}^{*t} \delta_x$ so that
\begin{align*}
\Big\lvert \op{L}^t f(x) - \int f \dd \mu \Big\rvert
  &= \Big\lvert \int f \dd \op{L}^{*t} \delta_x -\int f \dd\op{L}^{*t} \mu \Big\rvert \\
  &\le \Hol_\omega(f) \wass_\omega(\op{L}^{*t} \delta_x,\op{L}^{*t} \mu) \\ 
  &\le \Hol_\omega(f) F(t,\wass_\omega(\delta_x,\mu))\\
  &\le \Hol_\omega(f) F(t,\omega(\diam \Omega)).
\end{align*}
Similarly, for all $x,y\in\Omega$ we get
\begin{align*}
\Big\lvert \op{L}^t f(x) - \op{L}^t f(y) \Big\rvert
  &= \Big\lvert \int f \dd \op{L}^{*t}\delta_x - \int f \dd \op{L}^{*t}\delta_y \Big\rvert \\
  &\le \Hol_\omega(f) \wass_\omega(\op{L}^{*t}\delta_x,\op{L}^{*t}\delta_y) \\
  &\le \Hol_\omega(f) F(t,\omega\circ d(x,y))\\
  &\le \Hol_\omega(f) F^{\omega}_{\omega'}(t) \omega'\circ d(x,y).
\end{align*}
\end{proof}

It will prove convenient to first consider Dirac measure to prove decay, and the linearity of Wasserstein metric will enable us to extended our conclusions to all measures.
\begin{lemm}\label{lemm:linearity}
If $\op{L}^*$ is a linear operator on finite signed measures which preserves the set of probability measures and $\omega$ is a modulus of continuity, then for all $\mu,\nu\in\proba(\Omega)$ and all $\Pi\in\Gamma(\mu,\nu)$,
\[\wass_\omega(\op{L}^*\mu,\op{L}^*\nu) \le \int \wass_\omega(\op{L}^*\delta_x,\op{L}^*\delta_y) \dd\Pi(x,y).\]
\end{lemm}

\begin{proof}
For each $x,y\in\Omega$, let $\tilde\pi_{x,y}$ be an optimal plan between $\op{L}^*\delta_x$ and $\op{L}^*\delta_y$. Assume the map $(x,y)\mapsto \tilde\pi_{x,y}$ is measurable (this can be done, by a standard selection theorem, using that, $\Omega$ being compact, the set of optimal plans is always compact). By linearity of $\op{L}^*$ the measure
\[\tilde\Pi = \int \tilde\pi_{x,y} \dd\Pi(x,y)\]
is a coupling of $\op{L}^*\mu$ and $\op{L}^*\nu$, so that
\[
\wass_\omega(\op{L}^*\mu,\op{L}^*\nu) 
  \le \iint \omega\circ d(x_1,y_1) \dd\tilde\pi_{x,y}(x_1,y_1) \dd\Pi(x,y) 
  = \int \wass_\omega(\op{L}^*\delta_x,\op{L}^*\delta_y) \dd\Pi(x,y)
\]
\end{proof}

\subsection{Couplings of a transition kernel}\label{ssec:tight}

One prominent reason why expanding maps are easily amenable to the thermodynamical formalism is that their inverse branches are contracting. When studying non-uniformly expanding maps, one is led to relax this contracting condition on branches, which can be done in several ways. Since we are more generally interested in transition kernels, we will need a suitably defined contraction property for these objects, that should fit a given modulus of continuity $\omega$.

Let us first fix some notation. We are given a transition kernel $\mchain{M}=(m_x)_{x\in\Omega}$ on the compact metric space $\Omega$; given $t\in\mathbb{N}$ and $x\in\Omega$ we denote by $m_x^t$ the measure on $\Omega^t$ which is the law of a Markov chain $(X_1,...,X_t)$ starting at $X_0=x$ and following the transition kernel $\mchain{M}$. In other words, denoting by $\bar x=(x_1,\dots,x_t)$ points of $\Omega^t$, $m_x^t$ is defined by
\[\int_{\Omega^t} f(\bar x) \dd m_x^t(\bar x) = \int\cdots\iint f(\bar x) \dd m_{x_{t-1}}(x_t) \dd m_{x_{t-2}}(x_{t-1}) \cdots \dd m_x(x_1);\]
we also write 
\[(m_x^t)_{x\in \Omega} = \underbrace{(m_x)_{x\in \Omega} \circ (m_x)_{x\in \Omega} \circ \dots \circ (m_x)_{x\in \Omega}}_{t\mbox{ times}} \quad\text{or}\quad \mchain{M}^t = \mchain{M} \circ\dots\circ\mchain{M}.\]
Observe that $m_x^t$ is a probability measure on $\Omega^t$; the measure on $\Omega$ giving the law of the $n$-th element $X_n$ of the Markov chain with transition kernel $\mchain{M}$ started at $X_0=x$ is then $(e_t)_* m_x^t$, where $e_i:\Omega^t\to\Omega$ is the projection to the $i$-th factor (aka evaluation at time $i$) and the $*$ index denotes push forward.
It shall be observed that
\[(e_t)_* m_x^t = \op{L}_{\mchain{M},0}^{*t} \delta_x.\]

Given two probability measures $\bar\mu$, $\bar\nu$ on $\Omega^t$, we denote by $\Gamma(\bar\mu, \bar\nu)$ the set of probability measures $\Pi$ on $\Omega^t \times \Omega^t$ whose marginals are $\bar\mu$ and $\bar\nu$, i.e. such that 
\[p_{1*} \Pi = \mu \quad\mbox{and}\quad p_{2*}\Pi = \nu\]
where $p_1$, $p_2$ are the projection maps $\Omega^t\times\Omega^t \to\Omega^t$ on each factor. Then in particular, denoting by $(e_t,e_t) : \Omega_t\times \Omega_t \to \Omega\times \Omega$ the evaluation at time $t$, $(e_t,e_t)_* \Pi$ is a transport plan between $\mu_t:=(e_t)_*\bar \mu$ and $\nu_t:=(e_t)_*\bar \nu$. By abuse of notation, we will sometimes write $\wass_\omega(\bar\mu,\bar\nu)$ for $\wass_\omega(\mu_t,\nu_t)$.

\begin{defi}
A \emph{coupling} of a transition kernel $\mchain{M}=(m_x)_{x\in\Omega}$ is a family $\mchain{P}$ of probability measures $\Pi^t_{x,y} \in\Gamma(m_x^t,m_y^t)$ indexed by $t\in\mathbb{N}$ and $x,y\in\Omega$, such that for each $t$, the map $(x,y)\mapsto \Pi^t_{x,y}$ is Borel-measurable.

A \emph{coupling at time $i$} of $\mchain{M}$ is defined similarly, restricting to $t=i$. We sometimes call this a \emph{restricted coupling}, and whenever we want to make clear we mean a coupling defined for all times $t$, we use the term \emph{full coupling}.
\end{defi}
In other words, a (full) coupling records a way of pairing trajectories of two Markov chains following the same kernel $\mchain{M}$ but starting at possibly different points $x,y$. It would make sense to ask for a consistency condition, e.g. $(r_s)_* \Pi^t_{x,y} = \Pi^s_{x,y}$ whenever $s<t$, where $r_s:\Omega^t\to\Omega^s$ is the restriction to the first $s$ coordinates. However we will not actually need such a condition.

\begin{defi}\label{defi:decay}
A coupling $\mchain{P}=(\Pi^t_{x,y})_{t,x,y}$ (possibly restricted to a time $i$) of the transition kernel $\mchain{M}=(m_x)_{x\in\Omega}$ is said to be \emph{$\omega$-H\"older} when there is a constant $C$ such that for all $t$ (only $t=i$ in the restricted case) and all $x,y$,
\[ \int \omega \circ d(x_t,y_t) \dd\Pi^t_{x,y}(\bar x,\bar y) \le C \omega\circ d(x,y).\]
We also say that a full coupling is $\omega$-H\"older at step $i$ if its restriction to $t=i$ is $\omega$-H\"older.

A full coupling $\mchain{P}$ is said to have \emph{$\omega$-decay rate} $F$, where $F$ is a decay function, when for all $t$ and all $x,y$ it holds
\[ \int \omega\circ d(x_t,y_t) \dd\Pi^t_{x,y}(\bar x,\bar y) \le F(t,\omega\circ d(x,y)).\]
\end{defi}

\begin{defi}\label{defi:stableDecay}
Given a coupling $\mchain{P}$ and a normalized potential $A:\Omega\to\mathbb{R}$, we define $\mchain{P}_A =\big(e^{A^t(\bar x)} \dd \Pi_{x,y}^t(\bar x, \bar y) \big)_{t,x,y}$ and we extend the previous definition by saying that $\mchain{P}_A$ has \emph{$\omega$-decay rate} $F$ when for some constant $C=C(A)$, for all $t$ and for all $x,y$
\[ \int \omega\circ d(x_t,y_t) e^{A^t(\bar x)} \dd\Pi^t_{x,y}(\bar x,\bar y) \le  C F(t,\omega\circ d(x,y)).\]
$\mchain{P}$ is then said to have \emph{stable $\omega$-decay rate} $F$ when the above holds for all bounded, normalized $A$.
\end{defi}
The role of the constant $C$ is to allow a factor depending on $A$, but the decay rate we will consider (exponential or polynomial) are all defined up to a constant anyway. Observe that since $A$ is assumed to be normalized,
\[\int e^{A^t(\bar x)} \dd\Pi^t_{x,y}(\bar x,\bar y) = \int e^{A^t(\bar x)} \dd m_x^t(\bar x) = 1 \quad \forall t,x,y\]
so that for each $(t,x,y)$, $e^{A^t(\bar x)} \dd \Pi_{x,y}^t(\bar x, \bar y)$ is a probability measure; however Definition   \ref{defi:stableDecay} really is an extension of Definition \ref{defi:decay} since $\mchain{P}_A$ is not a coupling: its first marginal is the Markov chain $\mchain{M}_A= \big(e^{A(x_1)}\dd m_x(x_1)\big)_x$ but its second marginal is different, and might not even be a Markov chain (in $\mchain{P}_A$ the weight in the pairing of a $\bar x$ with a $\bar y$ is given by $A^t(\bar x)$, independently of $\bar y$). A sufficient condition to have stable decay is given in Section \ref{ssec:stableCriterion}.

Let $\mchain{M}=(m_x)_{x\in\Omega}$ be a transition kernel, $\op{L}_0=\op{L}_{\mchain{M},0}$ be the unweighted transfer operator, $\omega$ be a continuity modulus and $F$ be a decay function. Then concavity ensures that checking decay only for Dirac masses suffices to obtain decay for all probability measures.
\begin{lemm}\label{lemm:DiracReduction}
Let $\mchain{P}=(\Pi_{x,y}^t)_{x,y,t}$ be a coupling of $\mchain{M}$, $t\in\mathbb{N}$ and $\mu,\nu\in\proba(\Omega)$. If $\mchain{P}$ is $\omega$-H\"older with constant $C$, then
\[\wass_\omega(\op{L}_0^{*t} \mu , \op{L}_0^{*t} \nu) \le C\wass_\omega(\mu,\nu).\]
If $\mchain{P}$ has $\omega$-decay rate $F$, then
\[ \wass_\omega(\op{L}_0^{*t} \mu , \op{L}_0^{*t} \nu) \le F(t,\wass_\omega(\mu,\nu))\]
(in other words, $\op{L}_0^*$ also has decay rate $F$ in the metric $\wass_\omega$.)
\end{lemm}

\begin{proof}
We prove the second part, the first being the same, only simpler.
Let $\Pi^0$ be an optimal coupling of $\mu$ and $\nu$ for the cost $\omega\circ d$; then 
\[\Pi := \int_{(x,y)\in\Omega\times \Omega} \big( (p_t,p_t)_*\dd\Pi_{x,y}^t \big) \dd\Pi^0(x,y) \in \proba(\Omega\times \Omega)\]
is a coupling of $\op{L}_0^{*t} \mu$ and $\op{L}_0^{*t} \nu$, so that we have
\begin{align*}
\wass_\omega(\op{L}_0^{*t} \mu , \op{L}_0^{*t} \nu)
  &\le \int \omega\circ d(x_t,y_t) \dd\Pi(x_t,y_t) \\
  &= \iint \omega\circ d(x_t,y_t) \dd \Pi_{x,y}^t(\bar x,\bar y) \dd\Pi^0(x,y) \\
  &\le \int  F(t,\omega\circ d(x,y)) \dd\Pi^0(x,y) \\
  &\le F\Big(t,\int \omega\circ d(x,y) \dd\Pi^0(x,y)\Big) \\
  &\le F(t,\wass_\omega(\mu,\nu))
\end{align*}
where in the penultimate line we used that $F(t,\cdot)$ is concave.
\end{proof}

\subsection{Flat potentials}\label{ssec:flat}

Given a transition kernel $\mchain{M}=(m_x)_{x\in\Omega}$, 
finding a coupling  $\mchain{P}=(\Pi^t_{x,y})_{t,x,y}$ with a good decay gives a bound of the form
\[\wass_\omega(m_x^t,m_y^t) \le F(t, \omega\circ d(x,y) )\]
which will translate into a similar control for the unweighted operator $\op{L}_{\mchain{M},0}$ acting on $\C^\omega(\Omega)$. This is very simple to achieve, but to extend this to a weighted operator $\op{L}_A=\op{L}_{\mchain{M},A}$, 
we will need to ask the potential $A$ to have some special quality. 

\begin{defi}
Assume a coupling $\mchain{P}$ is fixed for the transition kernel $\mchain{M}$.
We say that a potential $A\in\C^\omega(\Omega)$ is \emph{flat} (with respect to $\mchain{P}$ and $\omega$) whenever for some constant $C>0$, for all $t\in\mathbb{N}$, all $x,y\in\Omega$ and $\Pi_{x,y}^t$-almost all $(\bar x,\bar y)$ it holds
\[\big\lvert A^t(\bar x) - A^t(\bar y) \big\rvert \le C \omega\circ d(x,y).\]
\end{defi}
The name has been chosen with the case of a map with a neutral point in mind: flatness is related to small variations near the neutral point.

We will see that this condition enables us to prove the Ruelle-Perron-Frobenius theorem for $\op{L}_A$ if the coupling is $\omega$-H\"older at step $1$, and to prove in addition a spectral gap if the coupling has exponential $\omega$-decay.

\begin{lemm}\label{lemm:unibound}
For all flat normalized potential $A\in\C^\omega(\Omega)$, there exist a constant $B>0$ such that for all $t$ and $\Pi_{x,y}^t$-almost all $(\bar x,\bar y)$,
\[e^{A^t(\bar x)-B\omega\circ d(x,y)} \le e^{A^t(\bar y)} \le e^{A^t(\bar x)+B\omega\circ d(x,y)}.\]
In particular, there exist a constant $C>0$ such that 
\[\frac1C e^{A^t(\bar x)} \le e^{A^t(\bar y)} \le C e^{A^t(\bar x)}.\]
\end{lemm}

\begin{proof}
By flatness:
\[e^{A^t(\bar y)} \le e^{\lvert A^t(\bar y)-A^t(\bar x)\rvert} e^{A^t(\bar x)}\le e^{B\omega\circ d(x,y)} e^{A^t(\bar x)} \le \underbrace{e^{B\omega(\diam \Omega)}}_{C} e^{A^t(\bar x)}\]
and similarly $e^{A^t(\bar y)}\ge e^{-B\omega\circ d(x,y)} e^{A^t(\bar x)}$.
\end{proof}

See Section \ref{sec:flatcrit} for criteria ensuring flatness. Basically, flatness will be shown under a (possibly weak) contraction property by choosing suitably two moduli of continuity: one to control the potential, and the second one to control the Birkhoff sums. The weaker the contraction property, the more gap we will have to take between the two moduli.

\section{A Ruelle-Perron-Frobenius theorem}\label{sec:positive}

In this Section, we slightly extend the proof given in \cite{KLS}, itself strongly inspired from Denker-Urbanski \cite{DenkerUrbanski:1991b}, for the existence of a positive eigenfunction for the spectral radius, in as great a generality we could achieve without obscuring the main point. Then we show how to apply this in our present setting.

We consider a Banach algebra $\fspace{X}$ of continuous, real valued  functions on a compact metric space $\Omega$. The norm on $\fspace{X}$ will be denoted by $\lVert\cdot\rVert$ while the uniform norm is $\lVert\cdot\rVert_\infty$, and we assume that $\lVert\cdot\rVert \gtrsim \lVert\cdot\rVert_\infty$ and that $\one\in\fspace{X}$. 

\begin{prop}\label{prop:eigenf}
Assume that the closed balls of $\fspace{X}$ are sequentially compact in the uniform norm, and let $\op{L}:\fspace{X}\to\fspace{X}$ be a linear operator preserving the cone of positive functions, continuous in both norms $\lVert\cdot\rVert$ and $\lVert\cdot\rVert_\infty$, such that for some constant $C>0$:
\begin{enumerate}
\item\label{enumi:eigenf1} for some $x\in\Omega$ and all $n$ it holds $1/C \le \big(\op{L}^n\one(x)\big)^{\frac{1}{n}}\le C$,
\item\label{enumi:eigenf2} for all $x,y\in\Omega$ and all $n\in\mathbb{N}$ it holds
$\op{L}^n \one(x) \le C \op{L}^n \one(y)$,
\item\label{enumi:eigenf3} for all $n$ it holds $\lVert \op{L}^n\one\rVert \le C \rVert \op{L}^n\one\rVert_\infty$.
\end{enumerate}
Let $\rho$ be the spectral radius of $\op{L}$ in the uniform norm, i.e.
\[\rho = \lim_{n\to\infty} \left(\sup_{f\in\fspace{X}\setminus\{0\}} \frac{\lVert \op{L}^nf\rVert_\infty}{\lVert f\rVert_\infty}\right)^{\frac{1}{n}}.\]
Then there exists a positive function $h\in\fspace{X}$ such that
$\op{L}h=\rho h$.
\end{prop}

\begin{proof}
Note that \ref{enumi:eigenf2} implies that \ref{enumi:eigenf1} holds uniformly for all $x$, up to changing the constant $C$.
By \ref{enumi:eigenf1}, the number
\[\tilde\rho := \limsup_{n\to\infty} \big(\op{L}^n\one(x)\big)^{\frac{1}{n}}\]
is finite and non zero, and by \ref{enumi:eigenf2} it does not depend on the choice of $x$. Moreover we have $\tilde\rho\le \rho$. For all $f\in\fspace{X}$ such that $\lVert f\rVert_\infty=1$, it holds $-2<f<2$ and since $\op{L}$ preserves the cone of positive functions,  we get $-2 \op{L}^n\one <\op{L}^n f < 2 \op{L}^n\one$ for all $n$.
In particular $\lVert\op{L}^nf\rVert_\infty^{\frac{1}{n}} <2^{\frac{1}{n}} \rVert\op{L}^n\one\rVert_\infty^{\frac{1}{n}}$, so that $\rho=\tilde\rho$.

We fix $x_0\in\Omega$; the radius of convergence of the series $\sum_{n\ge 1} z^n \op{L}^n\one(x_0)$
is equal to $1/\rho$ (note that we will only consider real values of $z$). As proved in Denker-Urbanski \cite{DenkerUrbanski:1991b} we can chose a positive sequence $(a_n)_{n\in\mathbb{N}}$ such that $a_{n+1}/a_n\to1$ and the modified series
\[\sum_{n\ge 1} a_n z^n \op{L}^n\one(x_0)\]
still has radius of convergence $1/\rho$ and diverges at $z=1/\rho$. We then set
\[Q(z) = \sum_{n\ge 1} a_n z^n \op{L}^n\one (x_0) \quad\mbox{and}\quad h_z(x) = \frac{1}{Q(z)}\sum_{n\ge 1} a_n z^n \op{L}^n\one(x).\]
Assumption \ref{enumi:eigenf3} ensures that for all $z<1/\rho$, we have $h_z\in\fspace{X}$, and together with assumption \ref{enumi:eigenf2} it implies that $\lVert h_z\rVert$ is bounded independently of $z$. Since the closed balls of $\fspace{X}$ are assumed to be sequentially compact in the uniform norm, we can extract a sequence $z_m\to1/\rho$ such that $h_m:=h_{z_m}$ converges in the uniform norm to some $h\in\fspace{X}$. The $h_z$ are uniformly bounded away from $0$, thus $h$ is positive. 

Last we use that $Q(z_m)\to +\infty$ when $m\to\infty$ to deduce that $h$ is the desired eigenfunction.
Given any $\varepsilon>0$, let $N_\varepsilon$ be such that for all $n\ge N_\varepsilon$ we have $|a_{n-1}/a_n-1|\le \varepsilon$. Then for all $x\in\Omega$ we have:
\begin{align*}
\left\lvert\op{L} h(x) - \rho h(x)\right\rvert
  &\le \limsup_{m \to \infty} \lVert \op{L} h_m-\rho h_m\rVert_\infty \\
  &\le \limsup_{m \to \infty} \frac{1}{Q(z_m)} \left\lvert \sum_{n=1}^ {N_\varepsilon-2} a_n z_m^n \op{L}^{n+1}\one(x) - \sum_{n=1}^ {N_\varepsilon-1} a_n z_m^n \rho \op{L}^{n}\one(x) \right\rvert \\
  &\qquad + \limsup_{m\to\infty} \frac{1}{Q(z_m)} \left\lvert\sum_{n\ge N_\varepsilon}  (a_{n-1}z_m^{n-1} - a_n  z_m^{n} \rho) \op{L}^n\one(x) \right\lvert\\
  &\le \limsup_{m\to\infty} \frac{1}{Q(z_m)} \left\lvert\sum_{n\ge N_\varepsilon}  (a_{n-1}z_m^{n-1} - a_n  z_m^{n} \rho) \op{L}^n\one(x) \right\lvert\\
  &\le \limsup_{m \to \infty} \frac{1}{z_mQ(z_m)} \sum_{n\ge N_\varepsilon} \left|\frac{a_{n-1}}{a_n}  - \rho z_m\right| a_n z_m^n \op{L}^n \one (x) \\
  &\le \limsup_{m\to\infty} \frac{\varepsilon+\lvert 1-\rho z_m\rvert}{z_m} \cdot \frac{\sum_{n\ge N_\varepsilon}a_n z_m^n \op{L}^n \one (x)}{\sum_{n\ge 1}a_n z_m^n \op{L}^n \one (x_0)} \\
  &\le \varepsilon\rho C  
\end{align*}
hence, $\op{L}h= \rho h$.
\end{proof}

We tried to make what we really use apparent in the statement of Proposition \ref{prop:eigenf}, but we will only use it in the following particular case.
\begin{prop}\label{prop:eigenf2}
Let $\mchain{M}=(m_x)_{x\in\Omega}$ be a transition kernel on a compact metric space $\Omega$ and let $\omega$ be a modulus of continuity. Assume that we have a coupling $(\Pi^t_{x,y})_{x,y,t}$ which is $\omega$-H\"older at step $1$ and that $A\in\C^\omega(\Omega)$ is flat. Then Proposition \ref{prop:eigenf} applies to $\fX=\C^\omega(\Omega)$ and $\op{L}=\op{L}_{\mchain{M},A}$, in particular there exist a positive number $\rho$ and a positive function $h\in\C^\omega$ such that $\op{L}_{\mchain{M},A} h=\rho h$.
\end{prop}

The hypothesis on the coupling is very mild, but for general couplings flat potentials might be rare.

\begin{proof}
It follows from the Arzel\`a-Ascoli theorem that closed balls of $\C^\omega(\Omega)$ are sequentially compact in the uniform norm, and by its definition $\op{L}_{\mchain{M},A}$ preserves the cone of positive functions and acts continuously in the uniform norm; more precisely $\lVert \op{L}_{\mchain{M},A} f\rVert_\infty\le \lVert e^A\rVert_\infty \lVert f\rVert_\infty$.

Since $\C^\omega(\Omega)$ is a Banach algebra (in particular $e^A\in\C^\omega(\Omega)$), using 
\[\int g(x_1) \dd m_x(x_1) = \int g(x_1) \dd\Pi^1_{x,y}(x_1,y_1)\]
we get for all $f\in\C^\omega(\Omega)$ and all $x,y\in\Omega$:
\begin{align*}
\big\lvert \op{L}_{\mchain{M},A} f(x) - \op{L}_{\mchain{M},A} f(y) \big\rvert 
  &= \Big\lvert \int e^{A(x_1)}f(x_1) \dd \Pi^1_{x,y}(x_1,y_1) - \int e^{A(y_1)}f(y_1) \dd \Pi^1_{x,y}(x_1,y_1) \Big\rvert \\
  &\le \int \big\lvert e^{A(x_1)}f(x_1)- e^{A(y_1)}f(y_1) \big\rvert \dd\Pi^1_{x,y}(x_1,y_1) \\
  &\le \Hol_\omega(e^A f) \int \omega\circ d(x_1,y_1) \dd\Pi^1_{x,y}(x_1,y_1) \\
  &\le \lVert e^A\rVert_\omega \lVert f \rVert_\omega C \omega\circ d(x,y)
\end{align*}
It follows that $\Hol_\omega(\op{L}_{\mchain{M},A} f) \le C\lVert e^A\rVert_\omega  \lVert f\rVert_\omega$ and $\op{L}_{\mchain{M},A}$ acts continuously on $\C^\omega(\Omega)$.

We have left to check conditions \ref{enumi:eigenf1}, \ref{enumi:eigenf2} and \ref{enumi:eigenf3} of Proposition \ref{prop:eigenf}. We have 
\[\op{L}_{\mchain{M},A}^n \one(x) = \int e^{A^n(\bar x)} \dd m^n_x(\bar x)\]
and $A^n(\bar x)$ lies between $n\min A$ and $n\max A$, so that
$ e^{\min A} \le \big(\op{L}_{\mchain{M},A}^n\one(x) \big)^{\frac1n} \le e^{\max A} $ for all $x$, proving Condition \ref{enumi:eigenf1}. Using again that $\Pi_{x,y}^t$ has marginals $m_x^t$ and $m_y^t$,  we have for all $x,y$:
\begin{align*}
\big\lvert\op{L}_{\mchain{M},A}^n\one(x)-\op{L}_{\mchain{M},A}^n\one(y) \big\rvert
  &=\Big\lvert \int e^{A^n(\bar x)} \dd\Pi^n_{x,y}(\bar x, \bar y) - \int e^{A^n(\bar y)} \dd\Pi^n_{x,y}(\bar x, \bar y) \Big\rvert \\
  &\le \int \big\lvert e^{A^n(\bar x)} - e^{A^n(\bar y)} \big\rvert \dd\Pi^n_{x,y}(\bar x,\bar y)  \\
  &\le \int e^{A^n(\bar y)} \big\lvert e^{A^n(\bar x)-A^n(\bar y)} -1 \big\rvert \dd\Pi^n_{x,y}(\bar x,\bar y)\\
  &\le C \int e^{A^n(\bar y)} \big\lvert A^n(\bar x)-A^n(\bar y) \big\rvert \dd\Pi^n_{x,y}(\bar x,\bar y)\\
  &\le C' \omega\circ d(x,y) \int e^{A^n(\bar y)} \dd\Pi^n_{x,y}(\bar x,\bar y)
\end{align*}
where the last two lines use flatness of $A$. Since $\mchain{P}$ is a coupling of $\mchain{M}$, it follows
\[ \big\lvert \op{L}_{\mchain{M},A}^n\one(x) - \op{L}_{\mchain{M},A}^n\one(y) \big\rvert \le C' \omega\circ d(x,y) \op{L}_{\mchain{M},A} \one(y),\]
proving \ref{enumi:eigenf2} and \ref{enumi:eigenf3}.
\end{proof}

Proposition \ref{prop:eigenf2} is half of the RPF theorem, and the second half (existence of an eigenprobability) will be a consequence of our contraction results (or can be proved directly). As seen in Section \ref{ssec:transferop}, this half already enables us to normalize the potential $A$ into $\tilde A= A + \log h - \log h\circ T -\log \rho$.
Note that as soon as composition with $T$ preserves $\C^\omega(\Omega)$ (e.g. if $T$ is Lipschitz), the normalized potential $\tilde A$ is still in $\C^\omega(\Omega)$: up to changing $h$ by a constant factor, we can assume $0<h<2$ and then $\log h$ develops as a converging power series of $h$, thus $\log h\in\C^\omega(\Omega)$.

A crucial observation is that if the coupling $\mchain{P}$ is $\omega$-H\"older and $A$ is flat, then $\tilde A$ is also flat. Indeed, for all $t,x,y$ and $\Pi_{x,y}^t$-almost all $(\bar x, \bar y)$, we have
\begin{align*}
\lvert \tilde A^t(\bar x)-\tilde A^t(\bar y) \rvert
  &\le \lvert A^t(\bar x)-A^t(\bar y) \rvert \\
  &\qquad + \Big\lvert \sum_{k=1}^t \log h(x_k)-\log h(T(x_k)) - \log h(y_k)+\log h(T(y_k)) \Big\lvert \\
  &\le  \lvert A^t(\bar x)-A^t(\bar y) \rvert \\
  &\qquad + \Big\lvert \sum_{k=1}^t \log h(x_k)-\log h(x_{k-1}) - \log h(y_k)+\log h(y_{k-1}) \Big\lvert \\
  &\le  \lvert A^t(\bar x)-A^t(\bar y) \rvert + \lvert \log h(x_t) -\log h(x) -\log h(y_t) + \log h(y) \rvert \\
  &\le  \lvert A^t(\bar x)-A^t(\bar y) \rvert+ \Hol_\omega(\log h) \big(\omega\circ d(x_t,y_t) + \omega\circ d(x,y) \big) \\
  &\lesssim \omega\circ d(x,y).
\end{align*}

Let us gather these observations in a statement which reduces the problem to find out whether $\op{L}_{\mchain{M},A}$ has a spectral gap (or other decay properties) to the case of normalized potentials.
\begin{coro}\label{coro:normalization}
Let $T:\Omega\to\Omega$ be a map on a compact metric space $\Omega$, let $\mchain{M}=(m_x)_{x\in\Omega}$ be a backward random walk for $T$ and let $\omega$ be a modulus of continuity such that composition with $T$ preserves $\C^\omega(\Omega)$ (this is satisfied if $T$ Lipschitz).

Assume that we have a coupling $\mchain{P}=(\Pi^t_{x,y})_{t,x,y}$ which is $\omega$-H\"older and that $A\in\C^\omega(\Omega)$ is flat with respect to $\mchain{P}$.

Then there is a \emph{normalized} and \emph{flat} potential $\tilde A\in \C^\omega(\Omega)$ that differs from $A$ by a coboundary and a constant; it follows that $\op{L}_{\mchain{M},A}$ and $\op{L}_{\mchain{M},\tilde A}$ are conjugated up to a constant. In particular, the property to have a spectral gap is equivalent for $\op{L}_{\mchain{M},A}$ and $\op{L}_{\mchain{M},\tilde A}$.
\end{coro}

\section{The main contraction result}\label{sec:maincontraction}

Our core result is the following.
\begin{theo}\label{theo:core}
Let $\mchain{M}$ be a transition kernel on a compact metric space $\Omega$, and let $\omega$ be a modulus of continuity. Let $A\in\C^\omega(\Omega)$ be a flat, normalized potential and set $\op{L}=\op{L}_{\mchain{M},A}$. 

Assume $\mchain{M}$ admits a coupling $\mchain{P}$ such that $\mchain{P}_A$ has $\omega$-decay with decay function $F$ and corresponding half-life $\tau=\tau_\frac12:(0,+\infty)\to \mathbb{N}$.

Then there exist constants $C>0$ and $k\in\mathbb{N}$ such that for all $\mu,\nu\in\proba(\Omega)$ with $\wass_\omega(\mu,\nu)=:r$ it holds
\[\wass_\omega\big(\op{L}^{*k\tau(r/k)} \mu ,\op{L}^{*k\tau(r/k)} \nu \big) \le \frac12 \wass_\omega(\mu,\nu)\]
and 
\[\wass_\omega\big(\op{L}^{*t} \mu ,\op{L}^{*t} \nu \big) \le C \wass_\omega(\mu,\nu) \quad\forall t\in\mathbb{N}.\]
In particular:
\begin{itemize}
\item if $F$ is exponential then $\tau(r)$ is bounded, and so is $k\tau(r/k)$, so that $\op{L}_{\mchain{M},A}^*$ decays exponentially in the metric $\wass_\omega$, and $\op{L}_{\mchain{M},A}$ has a spectral gap on $\C^\omega(\Omega)$.
\item if $F$ is polynomial then $\tau(r)\le D/r^\alpha$ so that $k\tau(r/k)\le D'/r^\alpha$ and $\op{L}_{\mchain{M},A}^*$ decays polynomially, with the same degree.
\end{itemize}
\end{theo}
It follows from this result that  as soon as  $\mchain{P}_A$ decays to $0$, whatever the speed, $\op{L}_{\mchain{M},A}^*$ fixes a unique probability measure $\mu_A$, which thus must be the RPF measure of $A$ (so we get the second half of the RPF theorem for flat potentials).

In the set-up of our main results, there is only one reasonable candidate for the coupling $\mchain{P}$, and to apply Theorem \ref{theo:core} our main task will be to identify sufficient conditions on potentials ensuring flatness.

\begin{proof}
The proof mostly follows the strategy of \cite{KLS}, the extra generality coming mostly from the use of decay times and moduli of continuity.

\begin{step} construct a transport plan between $\op{L}^{*t}\delta_x$ and $\op{L}^{*t}\delta_y$.\end{step}
Here we need the normalization assumption, to ensure these two measures are both of the same mass.
Fix $t\in\mathbb{N}$, $x,y\in \Omega$ and observe that $\op{L}^{*t}\delta_x = (e_t)_*\big( e^{A^t} \dd m_x^t\big)$ where $e_t:\Omega^t\to \Omega$ is the projection to the last coordinates. We seek an efficient transport plan between  $\op{L}^{*t}\delta_x$ and $\op{L}^{*t}\delta_y$, and we will construct it as $(e_t,e_t)_{*}\Pi$ where $\Pi$ is a transport plan between $e^{A^t} \dd m_x^t$ and $e^{A^t} \dd m_y^t$. What we are given by the coupling $\mchain{P}$ is a transport plan $\Pi_{x,y}^t$ between $m_x^t$ and $m_y^t$, and we will modify it to take into account the $e^{A^t}$ factors. Define a function
\begin{align*}
a : \Omega^t\times \Omega^t &\to \mathbb{R} \\
  (\bar x,\bar y) &\mapsto \min \big( e^{A^t(\bar x)} , e^{A^t(\bar y)} \big)
\end{align*}
so that $a \dd \Pi_{x,y}^t$ is a positive measure whose marginals are less than $e^{A^t} \dd m_x^t$ and $e^{A^t} \dd m_y^t$, respectively. There must thus exist some positive measure $\Lambda$ on $\Omega^t\times \Omega^t$ such that
\[\Pi := a \dd \Pi_{x,y}^t + \Lambda\]
is a probability measure with marginals exactly $e^{A^t} \dd m_x^t$ and $e^{A^t} \dd m_y^t$. We want to bound above the $\omega$-cost of $\Pi$; the basic idea is that the first term will be small by the decay hypothesis, the second one will be small because $\Lambda$ has small mass.

\begin{step}Bound from above the mass of $\Lambda$.\end{step}
By Lemma \ref{lemm:unibound}, for $\Pi_{x,y}^t$ almost all $(\bar x,\bar y)$ and for some constant $B$,
\[a(\bar x,\bar y) \ge e^{A^t(\bar x)} e^{-B \omega\circ d(x,y)};\]
using that $A$ is normalized it comes that the total mass of $a\dd\Pi_{x,y}^t$ is at least $e^{-B\omega\circ d(x,y)}$. Since $\Omega$ has finite diameter, up to enlarging $B$ this total mass can bounded from below both by a constant $e^{-B}\in (0,1)$ and by $1-B\omega\circ d(x,y)$.
The total mass of $\Lambda$ is therefore bounded above as follows: 
\[ \int \one \dd \Lambda \le \min\big( B\omega\circ d(x,y) , 1-e^{-B}\big).\]

\begin{step}Bound the cost of $\Pi$ for a modified metric.\end{step}
We introduce a new modulus of continuity 
\[\omega' = \min \big( K\omega, \omega(\diam\Omega) \big)\]
where $K$ is a positive constant to be specified later on (independently of $x,y$). We have $\omega'\circ d(x,y) \ge \omega\circ d(x,y)$ for all $x,y\in \Omega$ and $\omega'\le K\omega$, so that $\omega\circ d$ and $\omega'\circ d$ are Lipschitz-equivalent metrics on $\Omega$, and as a consequence $\wass_\omega$ and $\wass_{\omega'}$ are Lipschitz-equivalent (with the same constants).

We decompose the cost as
\[\int \omega'\circ d(x_t,y_t)  \dd\Pi(\bar x,\bar y)
  = \int \omega'\circ d(x_t,y_t) a(\bar x,\bar y) \dd\Pi_{x,y}^t(\bar x,\bar y) + \int \omega'\circ d(x_t,y_t) \dd\Lambda(\bar x,\bar y).\]
For the first term we get from $\omega$-decay of $\mchain{P}_A$ (with $D$ only depending on $A$):
\begin{align*}
\int \omega'\circ d(x_t,y_t) a(\bar x,\bar y) \dd\Pi_{x,y}^t(\bar x,\bar y)
  &\le K \int \omega\circ d(x_t,y_t) e^{A^t(\bar x)} \dd\Pi_{x,y}^t(\bar x,\bar y) \\
  &\le DK\cdot F(t,\omega \circ d(x,y)) \\
  &\le DK\cdot F(t,\omega' \circ d(x,y)).
\end{align*}
For the second term, we distinguish two cases. If $\omega\circ d(x,y)\ge \omega(\diam \Omega)/K$, then $\omega'\circ d(x,y)= \omega(\diam \Omega) = \omega'(\diam \Omega)$ and we bound the mass of $\Lambda$ by $1-e^{-B}$, so that
\begin{align*}
\int \omega'\circ d(x_t,y_t) \dd\Lambda(\bar x,\bar y) 
  &\le (1-e^{-B}) \omega'(\diam \Omega) \\
  &\le (1-e^{-B})\omega'\circ d(x,y).
\end{align*}
If $\omega\circ d(x,y)\le \omega(\diam \Omega)/K$, then $\omega'\circ d(x,y)=K\omega\circ d(x,y)$ and we bound the mass of $\Lambda$ by $B\omega\circ d(x,y)$:
\begin{align*}
\int \omega'\circ d(x_t,y_t) \dd\Lambda(\bar x,\bar y)
  &\le B\omega\circ d(x,y) \omega'(\diam \Omega) \\
  &\le \frac{B\omega(\diam \Omega)}{K} \omega'\circ d(x,y).
\end{align*}

Choosing $K$ large enough to ensure $\frac{B\omega(\diam \Omega)}{K}\le 1-e^{-B}$, we get in both cases
\begin{equation}
\int \omega'\circ d(x_t,y_t) \dd\Pi(\bar x,\bar y) \le DK\cdot F(t,\omega' \circ d(x,y)) + (1-e^{-B}) \omega'(d(x,y)).
\label{eq:step4}
\end{equation}

\begin{step}\label{step:eqn2} $\wass_\omega\big(\op{L}^{*t} \mu ,\op{L}^{*t} \nu \big) \le C \wass_\omega(\mu,\nu) \quad\forall t\in\mathbb{N},\ \forall \mu,\nu\in\proba(\Omega).$ \end{step}

Since $F(t,r)\lesssim r$, the previous step implies in particular that $\Pi$, as a restricted coupling at time $t$, is $\omega'$-H\"older; but $\omega\le \omega'\le K\omega$ on $[0,\diam\Omega]$ so that $\Pi$ is also $\omega$-H\"older. Then the claim follows from Lemma \ref{lemm:DiracReduction}.

\begin{step}There exist $\theta_1\in(0,1)$ and $k_1\in\mathbb{N}$ such that for all $r$, all $x,y\in\Omega$ such that $\omega'\circ d(x,y)\ge r$ and all $t\ge k_1\tau(r/2^{k_1})$, 
\[\wass_{\omega'}(\op{L}^{*t}\delta_x,\op{L}^{*t}\delta_y) \le \theta_1 \omega'\circ d(x,y).\] \end{step}
We choose any $\theta_1 \in(1-e^{-B},1)$ and $k_1$ large enough to ensure $DK/2^{k_1} + (1-e^{-B})\le \theta_1$, and then apply \eqref{eq:step4} (note that $k_1\tau(r/2^{k_1})\ge \tau(r)+\tau(r/2)+\dots +\tau(r/2^{k_1})$).


\begin{step}There exist $\theta\in(0,1)$ and $k_2\in\mathbb{N}$ such that for all $r$, all $\mu,\nu\in\proba(\Omega)$ such that $\wass_{\omega'}(\mu,\nu) = r$ and all $t\ge k_2\tau(r/k_2)$,
\[\wass_{\omega'}(\op{L}^{*t}\mu,\op{L}^{*t}\nu) \le \theta \wass_{\omega'}(\mu,\nu).\] \end{step}

Choose any $\theta \in (\theta_1,1)$ and let $\eta>0$ be small enough to ensure $\theta_1+C\eta\le \theta$, where $C$ is the constant of Step \ref{step:eqn2}.
Let $\mu,\nu$ be any two probability measure and let $\Pi\in\Gamma(\mu,\nu)$ be optimal for $\wass_{\omega'}(\mu,\nu)=:r$. Define $s:=\eta r$ and $E:=\{(x,y) \mid \omega'\circ d(x,y) \ge s \}$. For all $t\ge k_1 \tau(s/2^{k_1})$, using Lemma \ref{lemm:linearity} it comes
\begin{align*}
\wass_{\omega'}(\op{L}^{*t}\mu,\op{L}^{*t}\nu )
  &\le \int \wass_{\omega'}(\op{L}^{*t}\delta_x,\op{L}^{*t}\delta_y ) \dd\Pi(x,y) \\
  &\le \int_E \wass_{\omega'}(\op{L}^{*t}\delta_x,\op{L}^{*t}\delta_y ) \dd\Pi(x,y) + \int_{\Omega\times\Omega\setminus E} \mkern-36mu \wass_{\omega'}(\op{L}^{*t}\delta_x,\op{L}^{*t}\delta_y ) \dd\Pi(x,y) \\
  &\le \theta_1 \int_E \omega'\circ d(x,y) \dd\Pi(x,y) + C \int_{\Omega\times\Omega\setminus E} \mkern-36mu \omega'\circ d(x,y) \dd\Pi(x,y) \\
  &\le \theta_1 \wass_{\omega'}(\mu,\nu) + C \eta r \\
  &\le \theta \wass_{\omega'}(\mu,\nu).
\end{align*}
It suffices to choose $k_2\ge 2^{k_1}/\eta$. 

\begin{step}conclude.\end{step}
We deduce that the $\theta$ decay time $\tau^{\omega'}_\theta(r)$ of $\op{L}^*$ with respect to $\wass_{\omega'}$ is at most $k_2\tau(r/k_2)$.

Then for all $n\in\mathbb{N}$,
\[\tau^{\omega'}_{\theta^n}(r) \le k_2\tau(r/k_2)+k_2\tau(\theta r/k_2) + \dots + k\tau(\theta^{n-1} r/k_2) \]
and taking $n$ large enough to ensure $\theta^n\le \frac{1}{2K}$ we get
\[\tau^{\omega'}_{\frac{1}{2K}}(r) \le k_2 n \tau(\theta^{n-1} r/k_2) \le k \tau(r/k) \quad \text{for some }k.\]
Now since $\wass_\omega\le \wass_{\omega'} \le K\wass_\omega$, the decay time for $\op{L}^*$ with respect to $\wass_\omega$ satisfies $\tau^\omega_{\frac12} \le \tau^{\omega'}_{\frac{1}{2K}}$, and we are done.
\end{proof}

\section{Weakly contracting $1$-to-$k$ transition kernels}\label{sec:chains}

This Section develops the tools we  will need next to apply our general results of Sections \ref{sec:positive} and \ref{sec:maincontraction} to the cases considered in the Introduction. We introduce a class of transition kernels that appear naturally as backward random walks for the maps of Theorems \ref{theo:mainHol}-\ref{theo:mainpol}; we construct natural couplings for this class of transition kernels, and give flatness criteria with respect to these natural couplings and moduli of continuity $\omega_{\alpha+\beta\log}$; to be able to apply these criteria we study the rate of decay of $c^n(r)$ as $n\to\infty$ for relevant families of weak contractions $c:[0,+\infty) \to[0,+\infty)$; and finally we state a half-general result encompassing the results of Sections \ref{sec:positive} and \ref{sec:maincontraction} in a way that is easy to reuse.

\subsection{A criterion for stable $\omega$-decay}\label{ssec:stableCriterion}

We start by giving a sufficient condition for certain couplings to have stable $\omega$-decay rate, ensuring applicability of Theorem \ref{theo:core} without restriction on $A$. We assume that $\mchain{P}$ is itself Markovian, i.e. is given by $(\Pi_{x,y}^t)_{x,y} = (\pi_{x,y})_{x,y} \circ \dots \circ (\pi_{x,y})_{x,y}$ for some $\pi_{x,y}\in\Gamma(m_x,m_y)$.

\begin{defi}
We say that a Markov transition kernel $(\pi_{x,y})_{x,y}$ on $\Omega^2$ is \emph{non-dilating and contracting with positive probability} when there exist $\lambda,p\in(0,1)$ such that for all $x,y\in\Omega$:
\begin{itemize}
\item for $\pi_{x,y}$-almost all $(x_1,y_1)$, $d(x_1,y_1)\le d(x,y)$,
\item there exist a set $E_{x,y}\subset\Omega^2$ such that for all $(x_1,y_1)\in E$, $d(x_1,y_1)\le \lambda d(x,y)$ and $\pi_{x,y}(E_{x,y})\ge p$,
\end{itemize}
\end{defi}

\begin{lemm}
Let $(\pi_{x,y})_{x,y}$ be a Markovian coupling of $\mchain{M}$ and $A$ be a bounded potential normalized with respect to $\mchain{M}$. If $(\pi_{x,y})_{x,y}$ is non-dilating and contracting with positive probability,
then $(e^{A(x_1)}\dd\pi_{x,y}(x_1,y_1))_{x,y}$ is non-dilating and contracting with positive probability.
\end{lemm}

\begin{proof}
Both points are essentially obvious: $e^{A(x_1)}\dd\pi_{x,y}(x_1,y_1)$ is absolutely continuous with respect to $\pi_{x,y}$ and $\int_{E_{x,y}} e^{A(x_1)}\dd\pi_{x,y}\ge p \min(e^A)>0$. For the claim to make sense though, one has to observe that $(e^{A}\pi_{x,y})_{x,y}$ is a Markov transition kernel, which follows from the normalization hypothesis:
\[\int \one e^{A(x_1)} \dd \pi_{x,y}(x_1,y_1) = \int \one e^{A(x_1)} \dd m_x(x_1)  =\one.\]
\end{proof}

\begin{lemm}
If $(\pi_{x,y})_{x,y}$ is a Markov transition kernel on $\Omega^2$ which is non-dilating and contracting with positive probability, then for some $\lambda\in (0,1)$, all $(x,y)\in\Omega^2$ and all $t\in\mathbb{N}$,
\[\int d(x_t,y_t) \dd \pi_{x,y}^t(\bar x,\bar y) \le \lambda^t d(x,y)\] 
\end{lemm}

Here we denoted by $\pi_{x,y}^t$ the iterates of the Markov transition kernel, so that when $(\pi_{x,y})_{x,y}$ comes from the coupling $\mchain{P}$, $\pi_{x,y}^t=\Pi_{x,y}^t$.

\begin{proof}
It suffices to consider the case $t=1$, then conclude by induction. But this follows from the definition: denoting by $\lambda_1$ the contraction factor of $\pi_{x,y}$ on $E_{x,y}$,
\begin{align*}
\int d(x_1,y_1) \dd \pi_{x,y}(x_1,y_1) &= \int_{E_{x,y}} d(x_1,y_1) \dd \pi_{x,y}(x_1,y_1) + \int_{\Omega^2\setminus E_{x,y}}  d(x_1,y_1) \dd \pi_{x,y}(x_1,y_1) \\
  &\le p\lambda_1 d(x,y) + (1-p) d(x,y)
\end{align*}
then we take $\lambda= p\lambda_1 +  (1-p)$.
\end{proof}

\begin{coro}\label{coro:markovian}
Assume $\mchain{P}$ is Markovian with a transition kernel that is non-dilating and contracting with positive probability. Then for all bounded normalized potentials $A$, $\mchain{P}_A$ has exponential decay with respect to $\omega_{\alpha+\beta\log}$ for all $\alpha\in(0,1)$ and all $\beta\in \mathbb{R}$; and $\mchain{P}_A$ has polynomial decay of degree $\beta$ with respect to $\omega_{\beta\log}$ for all $\beta>0$. 
\end{coro}

\begin{proof}
Let $A$ be a bounded, normalized potential. Then for all modulus of continuity $\omega$, some $\lambda\in(0,1)$, and all $x,y,t$:
\begin{align*}
\int \omega\circ d(x_t,y_t) \dd\big(e^{A^t}\Pi_{x,y}^t\big)(\bar x,\bar y) 
  &\le \omega \Big( \int d(x_t,y_t) \dd\big(e^{A^t}\Pi_{x,y}^t\big)(\bar x,\bar y) \Big) \\
  &\le \omega(\lambda^t d(x,y)).
\end{align*}
There is only left to observe that for all $\alpha'<\alpha$, $\omega_{\alpha+\beta\log}(\lambda^t r) \lesssim \lambda^{\alpha't} \omega_{\alpha+\beta\log}(r)$ and that 
\[\omega_{\beta\log}(\lambda^t r) = \frac{\omega_{\beta\log}(r)}{\big(1+ t\cdot \omega_{\beta\log}(r)^{\frac1\beta} \log\frac1\lambda\big)^\beta}\]
which provides exactly the polynomial decay of degree $\beta$.
\end{proof}

\subsection{Weakly contracting $1$-to-$k$ transition kernels}

Let $\mchain{M}=(m_x)_{x\in\Omega}$ be a transition kernel of the form $m_x = \sum_{j=1}^k \frac1k \delta_{x^j}$ where for each $x\in\Omega$, $B(x)=(x^j)_{1\le j\le k}$ is a family of $k$ points. We will sometimes call $\mchain{M}$ a $1$-to-$k$ transition kernel, but note that some $x^j$ might coincide. Note that compared to the setup of \cite{KLS}, this is just another way to encode a $k$-multiset.

\begin{defi}
Let us call \emph{contraction function} any continuous 
map $c:[0,+\infty)\to [0,+\infty)$ such that $c(0)=0$ and $c(r)<r$ for all $r>0$.
We say that $\mchain{M}$ is \emph{weakly contracting} if there are a contraction function $c$ and a real number $\lambda>1$ such that, for all $x,y\in\Omega$, there exist permutations $\eta,\sigma$ of $\{1,\dots, k\}$ such that
\begin{enumerate}
\item\label{enumi:wcontract} for all $j\in\{1,\dots, k\}$ it holds $d(x^{\eta(j)},y^{\sigma(j)}) \le c(d(x,y))$,
\item $d(x^{\eta(k)},y^{\sigma(k)}) \le d(x,y)/\lambda$.
\end{enumerate}
\end{defi}
Most usually, we will have $c(r)\sim r$ when $r\to 0$ so that  the second contraction property will be the stronger one; it will often be satisfied for more $j$'s than just $j=k$, but the above is sufficient for our purposes.

\begin{defi}
When $\mchain{M}$ is weakly contracting, we define a \emph{natural coupling} $\mchain{P}$ as follows (it need not be unique given $\mchain{M}$, but is uniquely defined by the family of permutations $\eta$ and $\sigma$ above).
For each $(x,y)\in\Omega\times\Omega$, we fix permutations $\eta,\sigma$ realizing item \ref{enumi:wcontract} above (the dependency of $\eta,\sigma$ on both $x$ and $y$ will be kept implicit); then for each pair $(x,y)$ and each word $w = (j_1,\dots, j_t) \in \{1,\dots,k\}^t$ we let $\bar x_t^{w}, \bar y_t^{w} \in\Omega^t$ be the sequences $(x_1,\dots, x_t), (y_1,\dots, y_t)$ such that $x_1=x^{\eta(j_1)} \in B(x)$ and $y_1=y^{\sigma(j_1)}\in B(y)$, and for all $n$: $x_{n+1} = (x_n)^{\eta_n(j_n)} \in B(x_n)$ and $y_{n+1} = (y_n)^{\sigma_n(j_n)} \in B(y_n)$ where  $\eta_n$ and $\sigma_n$ are the permutation associated to the pair $(x_n,y_n)$. Then the natural coupling is
\[\Pi^t_{x,y} = \sum_{w\in \{1,\dots,k\}^t} \frac{1}{k^t} \delta_{(\bar x_t^{w}, \bar y_t^{w})}.\]
\end{defi}
In other words we pair together the orbits according to the pairing given in the definition of ``weakly contracting''.

The following lemma will be fed into Theorem \ref{theo:core}, and we see that it is really the regularity of observables that drives the speed of decay.
\begin{lemm}\label{lemm:natural}
Let $\mchain{M}$ be a weakly contracting $1$-to-$k$ transition kernel and $\mchain{P}$ be the natural coupling. For all bounded normalized potential $A$, $\mchain{P}_A$  has exponential decay with respect to $\omega_{\alpha+\beta\log}$, for all $\alpha\in (0,1)$ and all $\beta\in\mathbb{R}$, and $\mchain{P}_A$ has polynomial decay of degree $\beta$ with respect to $\omega_{\beta\log}$ for all $\beta>0$.
\end{lemm}

\begin{proof}
The natural coupling is Markovian, with transition kernel $\pi_{x,y}=\sum_j \frac1k\delta_{(x^{\eta(j)},y^{\sigma(j)})}$, which is non-dilating and contracting with positive probability thanks to the hypothesis that $\mchain{M}$ is weakly contracting. Corollary \ref{coro:markovian} provides the conclusion.
\end{proof}

Note that the above holds regardless of the contraction function (and would even hold with $c(r)=r$): a single contracting branch is sufficient to ensure the natural coupling of the transition kernel has good decay. By contrast, the specific features of $c$ will prove important for flatness: for weak $c$ our method will only apply to very regular potentials.

\subsection{Flatness criteria}
\label{sec:flatcrit}

\begin{lemm}\label{lemm:flat}
Let $\omega$, $\tilde\omega$ be moduli of continuity such that $\tilde\omega\le C\omega$ on $[0,1]$, so that $\C^{\tilde\omega}(\Omega) \subset \C^{\omega}(\Omega)$.

If $\mchain{M}$ is weakly contracting with contraction function $c$ and if there exist a constant $C$ such that for all $r\in(0,\diam\Omega]$,
\[\sum_{n\ge 1} \tilde\omega(c^n(r)) \le C\omega(r) \]
then every potential $A\in \C^{\tilde\omega}(\Omega)$ is flat with respect to $\mchain{P}$ and $\omega$.
\end{lemm}

\begin{proof}
Fix a potential $A\in \C^{\tilde\omega}(\Omega)$.
$\Pi^t_{x,y}$ is supported on pairs of sequences $(\bar x = (x_1,\dots,x_k),\bar y = (y_1,\dots,y_k))$ such that for all $n\le k$ it holds $d(x_n,y_n)\le c^n(d(x,y))$. For all $x,y,t$ and $\Pi_{x,y}^t$-almost all $\bar x=(x_1,\dots,x_t)$ ,$\bar y=(y_1,\dots,y_t)$ we thus have
\begin{align*}
\big\lvert A^t(\bar x) - A^t(\bar y) \big\rvert 
  &\le \sum_{n=1}^t \lvert A(x_n) - A(y_n) \rvert \\
  &\le \Hol_{\tilde \omega}(A) \sum_{n=1}^t \tilde\omega\circ d(x_n,y_n) \\
  &\le \Hol_{\tilde \omega}(A) \sum_{n=1}^t \tilde \omega (c^n(d(x,y))) \\
  &\lesssim \omega \circ d(x,y),
\end{align*}
which is the definition of flatness.
\end{proof}

In some circumstances, we will need a slightly finer analysis.
\begin{lemm}\label{lemm:flatbis}
Let $\mchain{M}$ be a weakly contracting $1$-to-$k$ transition kernel with contraction function $c$ and ratio $\lambda>1$. As above, $x^1,\dots,x^k$ denote the points supporting $m_x$ and $\eta,\sigma$ (the dependence on $(x,y)$ being kept implicit) define the natural coupling $\mchain{P}$.

Assume that there is a set $N\subset \Omega$ such that for all $x,y\in \Omega$ and all $j$, if either $x^{\eta(j)}\notin N$ or $y^{\sigma(j)}\notin N$, then $d(x^{\eta(j)},y^{\sigma(j)})\le d(x,y)/\lambda$. In other words, in the definition of ``weakly contracting'', the indices for which the stronger contraction does not hold are detected by the condition $(x^{\eta(j)},y^{\sigma(j)}) \in N^2$.

Let $\alpha\in(0,1]$ and $A\in\C^\alpha(\Omega)$.
If for some $C$, all $t$, all $x,y\in\Omega$ and $\Pi_{x,y}^t$-almost all $(\bar x,\bar y)$ staying in $N$ (i.e. $(x_n, y_n)\in N^2$ for all $n\in \{1,\dots, t\}$) we have
\[ \Big\lvert \sum_{n=1}^t A(x_n) - A(y_n) \Big\rvert \le C d(x,y)^\alpha \]
then $A$ is $\omega_{\alpha}$-flat.
\end{lemm}

\begin{proof}
Given $t,x,y$ and $\Pi_{x,y}^t$ almost any $(\bar x,\bar y)$, we decompose the time interval into ``runs in $N$'', i.e. write
$\{1,\dots,t\} = I_1\cup \dots \cup I_k$
where $I_j=\{n_j,\dots, n_{j+1}-1\}$ are intervals of integers, with $(n_j)$ is an increasing sequence and:
\begin{itemize}
\item $(x_{n+1},y_{n+1})\in N^2$ whenever $n$ and $n+1$ both lie in some $I_j$,
\item $d(x_{n_{j+1}},y_{n_{j+1}})\le d(x_{n_{j+1}-1},y_{n_{j+1}-1})/\lambda$ for all $j$.
\end{itemize}
We use the hypothesis and that distance decreases as $n$ increases to get
\begin{align*}
\Big\lvert \sum_{n=1}^t A(x_n)-A(y_n) \Big\rvert
  &\le  \sum_{j=1}^{k} \Big\lvert \sum_{n=n_j}^{n_{j+1}-1} A(x_n)- A(y_n) \Big\rvert \\
  &\le \sum_{j=1}^{k} C d(x_{n_j-1},y_{n_j-1})^\alpha \\
  &\le C \sum_{j=1}^{k} d(x,y)^\alpha/\lambda^{(j-1)\alpha}\\
  &\le C' d(x,y)^\alpha.
\end{align*}
\end{proof}

\subsection{Decay of iterated contraction functions}

In order to use Lemma \ref{lemm:flat}, we shall determine the speed of convergence to $0$ of iterates of $c$ in the cases of interest. Recall that $c$ is assumed to be a contraction function: it is continuous, $c(0)=0$ and $c(r)<r$ for all $r>0$.

\begin{lemm}\label{lemm:cdecay}
If for some $q\in(0,1)$ and $D>0$ we have \[c(r) = (1-Dr^q) r + o_{r\to0}(r^{1+q}) \]
then for some $a>0$, all $r\in (0,1]$ and all $n\in\mathbb{N}$ we have
\[ c^n(r) \le \frac{1}{\big(a n +\frac{1}{r^q}\big)^{\frac1q} }. \] 
\end{lemm}

\begin{proof}
We have
\begin{align*}
\frac{1}{c(r)^q} &= \frac{1}{r^q}\big(1-Dr^q+o(r^q)\big)^{-q} \\
  &= \frac{1}{r^q}\big(1+qDr^q+o(r^q)\big) \\
  &= \frac{1}{r^q}+ qD + o(1).
\end{align*}
For some $a>0$ and $r_1\in(0,1]$ and all $r\in(0,r_1)$, we thus have
\begin{equation}
\frac{1}{c(r)^q} \ge \frac{1}{r^q}+a.
\label{eq:translation}
\end{equation}
Since $c(r)<r$ on $[r_1,1]$ and $c$ is continuous, Inequality \eqref{eq:translation} holds on $(0,1]$ up to taking a smaller $a$. Then by induction, for all $n$ it holds
\[\frac{1}{c^n(r)^q} \ge \frac{1}{r^q}+an,\]
which is the claim.
\end{proof}

\begin{lemm}\label{lemm:clogdecay}
If for some $q\in(0,+\infty)$ we have 
\[c(r) = \Big(1-\frac{1}{\big(\log \frac1r \big)^{q}} \Big) r + o_{r\to0}\Big(\frac{r}{\big(\log \frac1r \big)^{q}}\Big), \]
then
for some $a>0$, all $r\in (0,1]$ and all $n\in\mathbb{N}$ we have
\[ c^n(r) \le e^{-\big(an+(\log \frac1r )^{q+1} \big)^{\frac{1}{q+1}}}.\] 
\end{lemm}

\begin{proof}
We have for $r>0$ and arbitrary $\alpha>0$:
\begin{align*}
\log\frac{1}{c(r)} &= \log\frac1r -\log \Big(1-\frac{1}{\big(\log \frac1r \big)^{q}} + o\Big(\frac{1}{\big(\log \frac1r \big)^{q}}\Big) \Big) \\
  &= \log\frac1r + \frac{1}{\big(\log \frac1r \big)^{q}} + o\Big(\frac{1}{\big(\log \frac1r \big)^{q}}\Big)\\
\big( \log\frac{1}{c(r)} \big)^{\alpha} 
  &= \big( \log\frac1r \big)^{\alpha} \Big( 1+  \frac{1}{\big(\log \frac1r \big)^{q+1}} + o\Big(\frac{1}{\big(\log \frac1r \big)^{q+1}}\Big)\Big)^{\alpha} \\
  &=  \big( \log\frac1r \big)^{\alpha} + \alpha\big( \log\frac1r\big)^{\alpha-(q+1)}  + o\Big(\big( \log\frac1r\big)^{\alpha-(q+1)} \Big).
\end{align*}
Taking $\alpha=q+1$ it comes
$\big( \log\frac{1}{c(r)} \big)^{q+1} =  \big( \log\frac1r \big)^{q+1} + (q+1) + o(1)$
so that for some $a,r_0>0$ and all $r\in (0,r_0)$ we get
\[\big( \log\frac{1}{c(r)} \big)^{q+1} \ge \big( \log\frac1r \big)^{q+1} + a.\]
Up to adjusting $a$, the same also holds true on $[r_0,1]$ since $c$ is continuous and $c(r)<r$ for positive $r$. It follows that
for all $r\in (0,1]$ and all $n\in\mathbb{N}$:
\[\big( \log\frac{1}{c^n(r)} \big)^{q+1} 
  \ge \big( \log\frac1r \big)^{q+1} + an \]
and the result follows.
\end{proof}

\subsection{A half-general result}

Now we specialize the general results above to the case of $k$-to-$1$ maps, getting closer to Theorems \ref{theo:mainHol}-\ref{theo:mainpol}.
\begin{theo}\label{theo:nuem}
Let $T:\Omega\to\Omega$ be a Lipschitz $k$-to-$1$ map of a compact metric space $\Omega$, such that the uniform backaward random walk $\mchain{M}=(m_x)_{x\in\Omega}$ defined by $m_x=\frac1k \sum_{y\in T^{-1}(x)} \delta_y$ is weakly contracting.
Let $A\in\C^{\alpha+\beta\log}(\Omega)$ be a flat potential with respect to some modulus $\omega_{\alpha+\beta\log}$ and to the natural coupling $\mchain{P}$.

Then the transfer operator $\op{L}_A=\op{L}_{T,A}$ has a positive eigenfunction $h_A$ in $\C^{\alpha+\beta\log}(\Omega)$ with a positive eigenvalue $\lambda_A$; its dual $\op{L}_{A}^*$ has an eigenprobability $\nu_A$ for the same eigenvalue and $d\mu_A:=h_A \dd\nu_A$ defines a $T$-invariant probability (the RPF measure of $A$).

If $\alpha>0$, then moreover  $\op{L}_A$ has spectral gap in $\C^{\alpha+\beta\log}$: for some $C>0$, $\delta \in (0,1)$ and for all observable $f\in\C^{\alpha+\beta\log}(\Omega)$ such that $\int f\dd\mu_A=0$ we have
\[ \lVert \op{L}_A^t f \rVert_{\alpha+\beta\log} \le C (1-\delta)^t \lambda_A^t \lVert f\rVert_{\alpha+\beta\log}.\]
In this case, we thus have exponential decay of correlations: for some $C$ and all $f,g$ it holds
\[\Big\lvert \int f \cdot g\circ T^t \dd\mu_A -\int f \dd\mu_A \int g\dd\mu_A\Big\rvert \le C (1-\delta)^t \Hol_{\alpha+\beta\log}(f) \lVert g\rVert_{L^1(\mu_A)}.\]

Otherwise, i.e. if $\alpha=0$ and $\beta>0$, $\op{L}_{T,A}$ decays in the uniform norm as a degree $\beta$ polynomial: for some $C>0$ and all observable $f\in\C^{\beta\log}(\Omega)$ such that $\int f\dd\mu_A=0$ we have
\[ \lVert \op{L}_A^t f \rVert_{\infty} \le C \frac{\lVert f\rVert_{\beta\log}}{t^\beta} \lambda_A^t.\]
In that case, we thus have polynomial decay of correlation of degree $\beta$: for  some $C$ and all $f,g$ it holds
\[\Big\lvert \int f \cdot g\circ T^t \dd\mu_A -\int f \dd\mu_A \int g\dd\mu_A\Big\rvert \le C \frac{ \Hol_{\beta\log}(f) \lVert g\rVert_{L^1(\mu_A)}}{t^\beta}.\]
\end{theo}

\begin{proof}
Since $\mchain{M}$ is weakly contracting, the natural coupling is $\omega$-H\"older for all $\omega$. Since $T$ is Lipschitz, it preserves $\C^{\alpha+\beta\log}(\Omega)$ by composition, and we can apply Corollary \ref{coro:normalization}: there is a positive eigenvalue $\lambda_A$ and a positive eigenfunction $h_A\in \C^{\alpha+\beta\log}(\Omega)$, and the normalized potential $\tilde A=A+\log h_A-\log h_A\circ T-\log \lambda_A$ is flat and in $\C^{\alpha+\beta\log}(\Omega)$.

By Lemma \ref{lemm:natural}, the natural coupling $\mchain{P}$ is $\omega_{\alpha+\beta\log}$-decaying with a rate which is either exponential (if $\alpha>0$) or polynomial of degree $\beta$ (if $\alpha=0$ and $\beta>0$). Theorem \ref{theo:core} then shows that $\op{L}_{\tilde A}^*$ has at least decay rate $F$ which is either exponential (if $\alpha>0$) or polynomial of degree $\beta$ (if $\alpha=0$ and $\beta>0$), with respect to the Wasserstein metric $\wass_{\alpha+\beta\log}$. In particular it has a unique fixed probability $\nu_{\tilde A}$. It follows that $\op{L}_A^*$ has a unique eigenprobability $\dd\nu_A=\frac{1}{h_A}\dd\nu_{\tilde A}$, and $\mu_A= \nu_{\tilde A}$ is $T$-invariant (see Section \ref{ssec:transferop}).

Now we apply Proposition \ref{prop:WassDecay} to $\op{L}_{\tilde A}$. First, if $\alpha>0$ we can take $F(t,r) = C (1-\delta)^t r$ for some $C>0$ and $\delta\in (0,1)$ and we get
\[\lVert \op{L}_{\tilde A}^t f- \mu_A(f) \rVert_{\alpha+\beta\log} \le C \Hol_{\alpha+\beta\log}(f) (1-\delta)^t\]
where $C$ depends notably on $\diam\Omega$. Since $\op{L}_A^t f= \lambda_A^t h_A \op{L}_{\tilde A}^t(h_A^{-1} f)$, when $\mu_A(f)=0$ it comes
\[\lVert \op{L}_{A}^t f \rVert_{\alpha+\beta\log} \le C \lambda_A^t \Hol_{\alpha+\beta\log}(h_A^{-1} f) (1-\delta)^t \le C'(1-\delta)^t\lambda_A^t\lVert f\rVert_{\alpha+\beta\log}.\]
Note that we have to switch from the H\"older constant of $h_A^{-1} f$ to the H\"older \emph{norm } of $f$ since the norm is multiplicative, but not the constant. 

Second, if $\alpha=0$ and $\beta>0$, we can take $F(t,r) = \frac{Br}{\big( tr^{\frac1\beta} + b\big)^\beta}$ and we get
\[\lVert \op{L}_{\tilde A}^t f -\mu_A(f)\rVert_\infty \le \Hol_{\beta\log}(f) \frac{C}{t^\beta},\]
from which we deduce as above $\lVert \op{L}_{A}^t f\rVert_\infty \le C \frac{\lVert f\rVert_{\beta\log}}{t^\beta} \lambda_A^t$ when $\mu_A(f)=0$.

In both case, the decay of correlations follows in the classical way: one observes that $\op{L}_{\tilde A}^t(f\cdot g\circ T^t)=\op{L}_{\tilde A}^t(f)\cdot g$, assumes $\mu_A(f)=0$ by adding a constant, and then writes
\begin{align*}
\Big\lvert \int f \cdot g\circ T^t \dd\mu_A \Big\rvert
  &= \Big\lvert \int f \cdot g\circ T^t \dd\big(\op{L}_{\tilde A}^{*t}\mu_A\big) \Big\rvert \\
  &= \Big\lvert \op{L}_{\tilde A}^t(f\cdot g\circ T^t) \dd\mu_A \Big\rvert \\
  &= \Big\lvert \op{L}_{\tilde A}^t(f) \cdot g \dd\mu_A \Big\rvert \\
  &\le \lVert \op{L}_{\tilde A}^t f \rVert_\infty \int \lvert g\rvert \dd\mu_A,
\end{align*}
ending the proof.
\end{proof}

In the case $\alpha=0$, we could obtain decay in an intermediate generalized H\"older norm (between the uniform norm and $\lVert\cdot\rVert_{\beta\log}$) by using the full extent of Proposition \ref{prop:WassDecay}.

\begin{rema}\label{rema:nopartition}
Theorem \ref{theo:nuem} applies only when the number of inverse images is constant (``full branches''), and a natural question is whether the method applies more generally.

Since we rely on spaces of continuous functions, one cannot hope to deal with the traditional transfer operator (corresponding to the uniform backward random walk) when branches are not full. However, Section \ref{sec:maincontraction} applies as soon as one finds a backward random walk that does preserve continuity and contracts appropriately. In order to preserve continuity, one needs to jump only with small probability to a point located near the endpoint of its branch; it should be possible to carry this out in some non-Markovian examples.
\end{rema}

\section{Application to circle maps}\label{sec:proofmain}

\subsection{The Pomeau-Manneville family}

In this section we finish the proof of Theorems \ref{theo:mainHol} and \ref{theo:mainpol} by considering the Pomeau-Manneville family: given an exponent $q>0$, we consider on the circle $\mathbb{T}=\mathbb{R}/\mathbb{Z}$ parametrized by $[0,1)$ the map
\begin{align*}
T_q : \mathbb{T} &\to \mathbb{T} \\
  x &\mapsto \begin{cases} 
    \big(1+(2x)^q \big)x &\text{if }x\in [0,\frac12] \\ 
    2x-1 &\text{if } x\in [\frac12, 1)
             \end{cases}
\end{align*}
and we let $\mchain{M}$ be the corresponding $1$-to-$k$ backward random walk.

\subsubsection{Inverse branches and weak contraction}

A point $y\in [0,1)$ has two inverse images, denoted by $b_1(y)$ and $b_2(y)=(y+1)/2$ (the maps $b_1$ and $b_2$ are called \emph{branches} of $T_q$). We do not search for an explicit expression of $b_1(y)$, but we observe that
\begin{itemize}
\item $b_1$ is concave,
\item $b_1(x)\ge x/2$, and 
\item $b_1(y) = \big(1-(2y)^q\big)y + O(y^{2q+1})$.
\end{itemize}

Turning back to the notation of Section \ref{sec:chains}, to prove that $\mchain{M}$ is weakly contracting we consider $x,y\in\mathbb{T}$ and their inverse images $x^1=b_1(x), x^2=(x+1)/2$ and $y^1=b_1(y), y^2=(y+1)/2$. 

When a shortest path from $x$ to $y$ avoids $0$, i.e. when $d(x,y) = \lvert x-y\rvert$, we pair $x^1$ with $y^1$ and $x^2$ with $y^2$ (i.e. $\sigma=\eta=\Id$). On the one hand we have $d(x^2,y^2)=d(x,y)/2$; on the other hand, by concavity $d(x^1,y^1)\le b_1(d(x,y))$.

When the unique shortest path from $x$ to $y$ goes through $0$, we pair $x^1$ with $y^2$ and $x^2$ with $y^1$. Up to swapping $x$ and $y$ we assume that $y$ is closer to $1$ and $x$ closer to $0$, so that $d(x,y)=x+1-y$ and $y \ge 1/2$. Then on the one hand, since $y\ge 1/2$ the derivative of $T_q'$ is greater than $2$ on $[y,\frac12]$ and we have $d(x^2,y^1)\le d(x,y)/2$. On the other hand, setting $d=d(x,y)$ we have
\[d(x^1,y^2) = b^1(x)+\frac{1-y}{2} = \frac{d+x}{2}-2^q x^{q+1} +O(x^{2q+1}).\]
The $O(x^{2q+1})$ is also an $O(d^{2q+1})$, and $d$ being fixed the quantity $\frac{d+x}{2}-2^q x^{q+1}$ is increasing in $x$ when $d\le x_0:= 2^{-\frac{q+1}q}(q+1)^{-\frac1q}$. We deduce
\[d(x^1,y^2) \le (1-(2d)^q)d +O(d^{2q+1}).\] 

This proves that $\mchain{M}$ is weakly contracting, with coefficient $\lambda=1/2$ and function $c(r) = (1-(2r)^q)r + O(r^{2q+1})$.

\subsubsection{Case $q<1$.}

We start with the case $q<1$ of Theorem \ref{theo:mainHol}. Let thus $\alpha,\gamma\in (0,1)$ be such that $\gamma-\alpha\ge q$ and let $A\in \C^\gamma(\mathbb{T})$. Since $\gamma>\alpha$, we have $\omega_\gamma\le C\omega_\alpha$ and in particular $A\in\C^\alpha(\mathbb{T})$. By Lemma \ref{lemm:cdecay} for some $a>0$ and all $r>0$ we have $c^n(r) \le \big(an+r^{-q}\big)^{-\frac1q}$, so that (using $\gamma>q$):
\begin{align*}
\sum_{n\ge1} \omega_\gamma(c^n(r)) &\le \sum_{n\ge1} \big(an+r^{-q}\big)^{-\frac\gamma q} \\
  &\le  \int_0^\infty \big(ax+r^{-q}\big)^{-\frac\gamma q} \dd x \\
  &\lesssim \int_{r^{-q}}^\infty y^{-\frac{\gamma}{q}} \dd y \\
  &\lesssim (r^{-q})^{1-\frac{\gamma}{q}} =  r^{\gamma-q} \le r^\alpha.
\end{align*}
By Lemma \ref{lemm:flat}, we deduce that $A$ is flat with respect to $\omega_\alpha$. Then Theorem \ref{theo:nuem} applies, showing that $\op{L}_A$ satisfies a Ruelle-Perron-Frobenius Theorem with a spectral gap, and exponential decay of correlations.

For Theorem \ref{theo:mainpol}, we consider $\beta>1$ and $A\in\C^{q+\beta\log}(\mathbb{T})$ and have to prove that $A$ is $\omega_{(\beta-1)\log}$-flat; then Theorem \ref{theo:nuem} shows the Ruelle-Perron-Frobenius Theorem for $\op{L}_{T_q,A}$ and degree $\beta-1$ polynomial decay of correlations. Flatness follows again from Lemma  \ref{lemm:flat}:
\begin{align*}
\sum_{n\ge1} \omega_{q+\beta\log}(c^n(r)) &\le \sum_{n\ge1} \frac{1}{(an+r^{-q})\Big(\log\big( r_0 (an+r^{-q})^{\frac1q}\big)\Big)^{\beta}} \\
  &\le \int_0^\infty \frac{1}{(ax+r^{-q})\Big(\log\big( r_0 (ax+r^{-q})^{\frac1q}\big)\Big)^{\beta}} \dd x \\
  &\lesssim \int_{r^{-q}}^\infty \frac{1}{y (\log(r_0 y))^\beta} \dd y\\
  &\lesssim \frac{1}{\big(\log(r_0 r^{-q})\big)^{\beta-1}}\\
  &\lesssim \frac{1}{\lvert \log(r) \rvert^{\beta-1}}.
\end{align*}

\subsubsection{Case $q\ge 1$.}

The case $q\ge1$ of Theorem \ref{theo:mainHol} is dealt with in a similar way. Let $\alpha,\gamma$ be such that $\gamma>q$ and $\gamma+1-\alpha\ge q$ and let $A\in\C^\alpha(\mathbb{T})$ be differentiable near $0$ with $A'(r)=O_{r\to0}(r^\gamma)$. We will be done if we can apply Theorem \ref{theo:nuem}, for which we have left to check that $A$ is $\omega_\alpha$-flat. With Lemma \ref{lemm:flatbis} in mind, we observe that
there is a neighborhood $N$ of $0$ where $A'(r)\le C r^\alpha$ and outside which $T_q$ is uniformly expanding, so that for some $\lambda>1$ and all $x,y$, if $x_1,y_1$ are inverse images of $x,y$ which are paired by the natural coupling, either $d(x_1,y_1)\le d(x,y)/\lambda$ or $x_1,y_1\in N$ and $d(x_1,y_1)\le c(d(x,y))$. Lemma \ref{lemm:flatbis} shows that it suffices to check flatness for pairs of backward orbits $\bar x,\bar y$ such that $(x_n,y_n)\in N^2$ for all $n\in\{1,\dots, t\}$. Setting $r=d(x_1,y_1)$ we then have $d(x_n,y_n)\le (an+r^{-q})^{-\frac1q}$ and $\lvert x_n\rvert, \lvert y_n\rvert \le C n^{-\frac1q}$ for some $a,C>0$ and all $n$ (Lemma \ref{lemm:cdecay}). It follows from the Mean Value Theorem that
\begin{align*}
\Big\lvert \sum_{n=1}^t A(x_n)- A(y_n) \Big\rvert
  &\lesssim \sum_{n=1}^t n^{-\frac{\gamma}{q}}(an+r^{-q})^{-\frac1q} \\
  &\lesssim  \sum_{an\le r^{-q}} n^{-\frac{\gamma}{q}}(r^{-q})^{-\frac1q} + \sum_{an> r^{-q}}  n^{-\frac{\gamma}{q}}(an)^{-\frac1q}  \\
  &\lesssim \sum_{an\le r^{-q}} n^{-\frac{\gamma}{q}}r  +  \sum_{an> r^{-q}} n^{-\frac{\gamma+1}{q}} \\
  &\lesssim r+r^{\gamma+1-q} \qquad\text{(since }\gamma>q\text{)}\\
  &\lesssim d(x,y)^\alpha,
\end{align*}
finishing the proof of Theorem \ref{theo:mainHol}.

\subsection{Non-uniformly expanding maps with a weak neutral point}

Let us now prove Theorem \ref{theo:mainless-tangent}. We consider $q>0$ and the map
\begin{align*}
T_{q\log} : \mathbb{T} &\to \mathbb{T} \\
  x &\mapsto \begin{cases} 0 & \text{if } x=0 \\
             \big( 1+ (1-\log 2x)^{-q}\big)x & \text{if } x\in(0,\frac12]\\
             2x-1 &\text{if }x\in[\frac12,1)
             \end{cases}     
\end{align*}
whose uniform backward walk $\mchain{M}$ is weakly contracting, with contraction function of the form
\[c(r) = \Big(1-\frac{1}{\big(\log \frac1r \big)^{q}} \Big) r + o_{r\to0}\Big(\frac{r}{\big(\log \frac1r \big)^{q}}\Big).\]
Lemma \ref{lemm:clogdecay} tells us that for some $a>0$,
\[ c^n(r) \le e^{-\big(an+(\log \frac1r )^{q+1} \big)^{\frac{1}{q+1}}}.\] 
Given any $\alpha\in(0,1)$, for all $r>0$ we thus have:
\begin{align*}
\sum_{n\ge 1} \omega_{\alpha+q\log}(c^n(r))
  &\le \sum_{n\ge 1} \frac{e^{-\alpha\big(an+(\log \frac1r )^{q+1} \big)^{\frac{1}{q+1}}}}{\big(an+(\log \frac1r )^{q+1} \big)^{\frac{q}{q+1}}} \\
  &\le \int_0^\infty \frac{e^{-\alpha\big(ax+(\log \frac1r )^{q+1} \big)^{\frac{1}{q+1}}}}{\big(ax+(\log \frac1r )^{q+1} \big)^{\frac{q}{q+1}}} \dd x \\
  &\lesssim \int_{(\log\frac{1}{r})^{q+1}}^\infty \frac{e^{-\alpha y^{\frac{1}{q+1}}}}{y^{\frac{q}{q+1}}} \dd y\\
  &\lesssim e^{-\alpha \big((\log\frac{1}{r})^{q+1}\big)^{\frac{1}{q+1}}} = r^\alpha,
\end{align*}
so that Lemma \ref{lemm:flat} ensures that every potential
$A\in\C^{\alpha+q\log}(\mathbb{T})$ is flat with respect to $\omega_\alpha$ and the natural coupling $\mchain{P}$. Then Theorem \ref{theo:nuem} can be applied and Theorem \ref{theo:mainless-tangent} follows.

\subsection{General non-uniformly expanding maps}

We turn to the proof of Theorem \ref{theo:mainDense}.
 We consider $\alpha\in(0,1]$ and a $k$-to-$1$ map $T:\mathbb{T}\to\mathbb{T}$ with a neutral fixed point at $0$ and such that for all neighborhood $N$ of $0$, there is some $\lambda=\lambda(N)>1$ such that $|T'|\ge \lambda$ outside $N$. Taking $N$ small enough we can assume that for some $\lambda$ and all $x,y\in\mathbb{T}$, given the permutations $\sigma,\eta$ defining a pairing of their inverse images $x^1,\dots,x^k,y^1,\dots,y^k$ to nearest neighbors, we have $d(x^{\sigma(j)},y^{\eta(j)})\le d(x,y)/y$ except possibly for one pair $(x^{\sigma(1)},y^{\eta(1)})\in N^2$ (this can be considered a suitable interpretation of the hypothesis that $T$ is non-uniformly expanding in order to include non-derivable maps).

It follows at once from Lemma \ref{lemm:flatbis} that every potential $A\in\C^\alpha(\mathbb{T})$ which is constant on a neighborhood $N$ of $0$ is $\omega_\alpha$-flat, and Theorem \ref{theo:nuem} ensures that $\op{L}_A$ satisfies a Ruelle-Perron-Frobenius theorem with a spectral gap.

The fact that potentials which are constant in some neighborhood of $0$ are dense in $\C^\alpha(\mathbb{T})$ with respect to $\lVert\cdot\rVert_\gamma$ for all $\gamma\in(0,\alpha)$ is a classical fact, but let us give a proof nonetheless.

\begin{proof}[Proof of denseness of $V$]
For any $B\in \C^\alpha(\mathbb{T})$ and any $\varepsilon>0$, let $A$ be the potential such that $A(x)=B(0)$ for all $x$ at distance at most $\varepsilon/2$ from $0$, $A(x)=B(x)$ for all $x$ at least $\varepsilon$ away from $0$, and which is affine on both $[\varepsilon/2,\varepsilon]$ and $[1-\varepsilon,1-\varepsilon/2]$. Then $A$ is $\alpha$-H\"older with $\Hol_\alpha(A)\lesssim \Hol_\alpha(B)$, and for all $x\in\mathbb{T}$ we moreover have
\[\lvert A(x)-B(x)\rvert \lesssim \Hol_\alpha(B) \varepsilon^{\alpha-\gamma} d(x,0)^\gamma.\]
This already shows that $A$ and $B$ are close in the uniform norm. Let $x,y\in\mathbb{T}$ and first assume $d(x,y)\ge \varepsilon$. Then either $A$ and $B$ coincide at both $x$ and $y$, or they coincide at one of them ($y$, say) and $d(x,0)\le \varepsilon \le d(x,y)$; then
\[\lvert (A-B)(x) - (A-B)(y) \rvert \lesssim \Hol_\alpha(B) \varepsilon^{\alpha-\gamma} d(x,0)^\gamma \lesssim  \varepsilon^{\alpha-\gamma} d(x,y)^\gamma.\]
In the case when $d(x,y)\le \varepsilon$ we can estimate directly
\[\lvert (A-B)(x) - (A-B)(y) \rvert \lesssim \Hol_\alpha(B) d(x,y)^\alpha\lesssim \varepsilon^{\alpha-\gamma} d(x,y)^\gamma.\]
We deduce that $\lVert A-B\rVert_\gamma \lesssim \varepsilon^{\alpha-\gamma}$, concluding the proof.
\end{proof}

\subsection{Uniformly expanding maps}

We conclude by the proof of Theorem \ref{theo:mainUniform} and Corollary \ref{coro:mainAcim}.
Assume that $T:\Omega\to\Omega$ is a $k$-to-$1$ uniformly expanding map, by which we means that the uniform backward random walk $\mchain{M}=(m_x)_{x\in\Omega}$ defined by
\[m_x = \sum_{y\in T^{-1}(x)} \frac1k \delta_y\]
is $\lambda$ contracting for some $\lambda>1$ (which in our terminology is the same as being weakly contracting with contraction function $c(r)=r/\lambda$).
This definition in particular includes the classical definition, i.e. the case when $\Omega$ is a manifold and $\lVert D_xT(v) \rVert \ge \lambda \lVert v\rVert$ for all $x\in\Omega$ and all $v\in T_x\Omega$ (as briefly explained in \cite{KLS}, Example 2.9).

Consider a parameter $\beta>1$ and a potential $A\in \C^{\beta\log}(\Omega)$. We will again apply Lemma \ref{lemm:flat}, observing that
\begin{align*}
\sum_{n\ge 1} \omega_{\beta\log}(c^n(r)) 
  &= \sum_{n\ge 1} \frac{1}{\big(\log(\frac{r_0}{r}\lambda^{n})\big)^\beta} \\
  &\le \int_0^\infty \frac{1}{\big(x\log\lambda+\log\frac{r_0}{r}\big)^\beta} \dd x \\
  &\lesssim \int_{\log\frac{r_0}{r}} y^{-\beta} \dd y\\
  &\lesssim \big(\log\frac{1}{r}\big)^{1-\beta} \\
  &\lesssim \omega_{(\beta-1)\log}(r);
\end{align*}
it follows that $A$ is $\omega_{(\beta-1)\log}$-flat. Theorem \ref{theo:nuem} then implies Theorem \ref{theo:mainUniform}.

Now, assume further that $\Omega$ is a connected Riemannian manifold  and that $T$ is $\C^1$. If the Jacobian $JT$ is of regularity $\C^{\beta\log}$ for some $\beta>1$, then the ``natural'' potential $A=-\log JT$ can be applied Theorem \ref{theo:mainUniform}. But this potential is named natural for the reason that if $\nu$ denotes the Riemannian volume (normalized to have total mass $1$), then $\op{L}_A^*(\nu)=\nu$; from Theorem \ref{theo:mainUniform} we know that there is a positive eigenfunction $f\in \C^{(\beta-1)\log}(\mathbb{T})$, and $f\dd\nu$ is $T$-invariant (it is the eigenprobability for $\op{L}_{\tilde A}$ where $\tilde A$ is the normalized potential). The polynomial decay of correlations also follows from Theorem \ref{theo:mainUniform}, ending the proof of Corollary \ref{coro:mainAcim}.

\bibliographystyle{alpha}
\bibliography{WeaklyExp}
\end{document}